\documentclass{article}
\usepackage{amsmath,amssymb,amsthm,mathtools,enumitem,geometry,esint,color,hyperref,cleveref}

\setlist[enumerate,1]{label=(\roman*)}

\theoremstyle{plain}
\newtheorem{theo}{Theorem}[section]

\theoremstyle{definition}
\newtheorem{defi}[theo]{Definition}
\newtheorem*{defi*}{Definition}
\newtheorem*{exa*}{Example}

\newtheorem{lem}[theo]{Lemma}
\newtheorem*{lem*}{Lemma}
\newtheorem{cor}[theo]{Corollary}
\newtheorem{pro}[theo]{Proposition}
\newtheorem*{pro*}{Proposition}

\newtheorem*{cla*}{Claim}

\newtheorem*{lie*}{Lie}

\newtheorem*{alg*}{Algorithm}

\theoremstyle{remark}
\newtheorem{rem}[theo]{Remark}
\newtheorem*{rem*}{Remark}

\allowdisplaybreaks

\crefname{cor}{Corollary}{Corollaries}
\Crefname{cor}{Corollary}{Corollaries}
\crefname{pro}{Proposition}{Propositions}
\Crefname{pro}{Proposition}{Propositions}
\crefname{lem}{Lemma}{Lemmas}
\Crefname{lem}{Lemma}{Lemmas}
\crefname{theo}{Theorem}{Theorems}
\Crefname{theo}{Theorem}{Theorems}
\crefname{rem}{Remark}{Remarks}
\Crefname{rem}{Remark}{Remarks}
\crefname{exa}{Example}{Examples}
\Crefname{exa}{Example}{Examples}
\crefname{defi}{Definition}{Definitions}
\Crefname{defi}{Definition}{Definitions}
\crefname{section}{Section}{Sections}
\Crefname{section}{Section}{Sections}
\crefname{subsection}{Subsection}{Subsections}
\Crefname{subsection}{Subsection}{Subsections}
\crefname{equation}{}{}
\Crefname{equation}{Formula}{Formulas}
\crefname{enumi}{}{}
\Crefname{enumi}{Item}{Items}
\crefname{figure}{Figure}{Figures}
\Crefname{figure}{Figure}{Figures}

\numberwithin{equation}{section}

\renewcommand\tilde\widetilde
\renewcommand\hat\widehat
\newcommand\intd{\mathop{}\!\mathrm{d}}
\newcommand\df\emph
\newcommand\ind[1]{1_{#1}}
\newcommand\tx\text
\newcommand\f\frac
\newcommand\seq{:=}

\newcommand\loc{\mathrm{loc}}
\renewcommand\S{\mathcal S}
\newcommand\Q{\mathcal Q}
\renewcommand\P{\mathcal P}
\newcommand\B{\mathcal B}
\newcommand\C{\mathcal C}
\newcommand\F{\mathcal F}
\newcommand\R{\mathcal R}
\newcommand\T{\mathcal T}
\renewcommand\H{\mathcal H}
\DeclareMathOperator\var{var}
\DeclareMathOperator\diam{diam}
\DeclareMathOperator\dist{dist}
\DeclareMathOperator\conv{conv}
\let\div\relax
\DeclareMathOperator\div{div}
\newcommand\M{{\mathrm M}}
\newcommand\Mj{\overline{\mathrm M}}

\newcommand\Mu{\widetilde{\mathrm M}}

\newcommand\BV{\mathrm{BV}}

\newcommand\Sph{\mathbb S}
\newcommand\tb[1]{\mathop{\partial{}}{#1}}
\newcommand\mb[1]{\mathop{\partial_*}{#1}}
\newcommand\mbe[1]{\mb{(#1)}}
\newcommand\mbb[1]{\mb{\Bigl(#1\Bigr)}}
\newcommand\tbe[1]{\tb{(#1)}}
\newcommand\tbb[1]{\tb{\Bigl(#1\Bigr)}}

\newcommand\tc[1]{\overline{#1}}
\newcommand\mc[1]{\tc{#1}^*}
\newcommand\ti[1]{\mathring{#1}}

\newcommand\mi[1]{\ti{#1}^*}

\newcommand\sm[1]{\mathcal{H}^{d-1}(#1)}
\newcommand\smb[1]{\mathcal{H}^{d-1}\Bigl(#1\Bigr)}
\newcommand\lm[1]{\mathcal{L}(#1)}
\newcommand\lmb[1]{\mathcal{L}\Bigl(#1\Bigr)}
\newcommand\rad[1]{r(#1)}
\newcommand\ms[1]{\mathrm m(#1)}
\newcommand\msb[1]{\mathrm m\Bigl(#1\Bigr)}
\newcommand\ml[2]{\mathrm m(#1,#2)}

\newcommand\bll[1]{B(#1)}
\newcommand\cone[3]{\bcone{B(#1,#3)}{#2}}
\newcommand\bcone[2]{\conv(#1\cup\{#2\})}
\newcommand\bconeb[2]{\conv\Bigl(#1\cup\{#2\}\Bigr)}

\makeatletter
\newcommand{\oset}[3][0ex]{%
  \mathrel{\mathop{#3}\limits^{
    \vbox to#1{\kern-2\ex@
    \hbox{$\scriptstyle#2$}\vss}}}}
\makeatother

\newcommand\msubset{\oset[-.1ex]*\subset}

\newcommand\meq{\oset*=}
\newcommand\msubsetneq{\oset[-.2ex]*\subsetneq}

\begin{document}

\title{The Variation of the Uncentered Maximal Operator with respect to Cubes}
\author{Julian Weigt\footnote{
Aalto University,
Department of Mathematics and Systems Analysis,
P.O.\ Box 11100,
FI-00076 Aalto University,
Espoo,
Finland,
\texttt{julian.weigt@aalto.fi}
}
}

\maketitle

\begingroup
\begin{NoHyper}%
\renewcommand\thefootnote{}\footnotetext{%
}%
\addtocounter{footnote}{-1}%
\end{NoHyper}%
\endgroup

\begin{abstract}
We consider the maximal operator with respect to uncentered cubes
on Euclidean space with arbitrary dimension.
We prove that for any function with bounded variation,
the variation of its maximal function is bounded by the variation of the function
times a dimensional constant.

We also prove the corresponding result for maximal operators with respect to collections of more general sets than cubes.
The sets are required to satisfy a certain inner cone star condition and in addition the collection must enjoy a tiling property which for example the collection of all cubes does enjoy and the collection of all Euclidean balls does not.
\end{abstract}

\section{Introduction}
\label{sec_introduction}

For a locally integrable real-valued function \(f\in L^1_\loc(\mathbb{R}^d)\), with \(d\in\mathbb{N}\),
we consider the uncentered maximal function over cubes, defined by
\[
\M f(x)
=
\sup_{Q\ni x}
\f1{\lm Q}\int_Q
f(y)\intd y
,
\]
where the supremum is taken over all open axes-parallel cubes \(Q\) which contain \(x\in\mathbb{R}^d\).
We discuss maximal operators with respect to more general sets in \cref{sec_results}.
Usually, the maximal operator integrates over \(|f|\) instead of \(f\) because the maximal function is classically used for \(L^p\) estimates such as the Hardy-Littlewood maximal function theorem which states that for every \(p>1\) and \(f\in L^p(\mathbb{R}^d)\) we have
\begin{equation}
\label{eq_hlmft}
\lVert\M f\rVert_{L^p(\mathbb{R}^d)}
\leq C_{d,p}
\lVert f\rVert_{L^p(\mathbb{R}^d)}
.
\end{equation}
In this manuscript we discuss the regularity of maximal functions for which also sign-changing functions matter.
The first regularity result is due to Kinnunen who proved in \cite{MR1469106} that for \(p>1\) and \(f\in W^{1,p}(\mathbb{R}^d)\) the bound
\begin{equation}
\label{eq_goalp}
\lVert\nabla\M f\rVert_{L^p(\mathbb{R}^d)}
\leq C_{d,p}
\lVert\nabla f\rVert_{L^p(\mathbb{R}^d)}
\end{equation}
holds, from which it follows that the maximal operator is bounded on \(W^{1,p}(\mathbb{R}^d)\). 
Originally, \cref{eq_hlmft,eq_goalp} have been proven only for the Hardy-Littlewood maximal operator which averages over all balls centered in \(x\),
but both arguments work for a variety of maximal operators, including the operator \(\M\) defined above.
The proof strategy for \cref{eq_goalp} relies on \cref{eq_hlmft} and thus also fails for \(p=1\).
The question of whether \cref{eq_goalp} nevertheless holds with \(p=1\)
for any maximal operator has become a well known problem
and has been subject to considerable research.
However, it has so far remained mostly unanswered, except in one dimension.
Our main result is the following:
\begin{theo}
\label{theo_goal}
Let \(f\in L^1_\loc(\mathbb{R}^d)\) with \(\var f<\infty\).
Then \(\M f\in L_\loc^{\f d{d-1}}(\mathbb{R}^d)\) and
\begin{equation}
\label{eq_goal}
\var\M f
\leq C_d
\var f
\end{equation}
where the constant \(C_d\) depends only on the dimension \(d\).
\end{theo}
\Cref{theo_goal} answers a question in the paper \cite{MR2041705} of Haj\l{}asz and Onninen from 2004 for the uncentered maximal operator over cubes.
This question has originally been raised for the classical centered and the uncentered Hardy-Littlewood maximal operator given by
\[
\Mu f(x)
=
\sup_{B\ni x}
\f1{\lm B}
\int_B f(y)\intd y
,
\]
where the supremum is taken over all balls \(B\ni x\),
but remains valid for a wide range of maximal operators.
We prove the variation bound corresponding to \cref{eq_goal} also for maximal operators which average over more general sets than cubes, see \cref{theo_maininfinite,cor_manycases}.
The requirements on the collection of sets over which the maximal operator averages are twofold: 
we require the collection to have a certain tiling property, \cref{defi_decomposition},
and each individual set in the collection has to satisfy a certain inner cone star condition, \cref{defi_star}.
The collection of all cubes is the simplest example to have these two properties.
While a ball satisfies the inner cone star condition,
the collection of all balls does not have the required tiling property and therefore one part of our argument does not apply to the classical uncentered Hardy-Littlewood maximal operator
\(
\Mu
,
\)
see \cref{rem_balls}.

There is a subtle difference between \cref{eq_goalp} for \(p=1\) and \cref{theo_goal}:
\cref{theo_goal} is more generous in the condition on \(f\) as any \(f\in W^{1,1}(\mathbb{R}^d)\) satisfies \(\var f=\lVert\nabla f\rVert_1<\infty\).
However its conclusion is slightly weaker because \(\var\M f<\infty\) only means that \(\M f\) has a finite Radon measure as its weak gradient, not necessarily a function in \(L^1(\mathbb{R}^d)\), even if we assume \(f\in W^{1,1}(\mathbb{R}^d)\).
And indeed, some maximal functions covered by our more general \cref{theo_maininfinite}, such as the dyadic maximal operator which only averages over dyadic cubes, satisfy \cref{eq_goal} but do not possess a weak gradient in \(L^1(\mathbb{R}^d)\) no matter how smooth \(f\) is.
Lahti showed in \cite{panubvsobolev}
that if the variation bound \cref{eq_goal}
holds for the Hardy-Littlewood maximal operator \(\Mu\), then for any function \(f\in W^{1,1}(\mathbb{R}^d)\)
we have \(\nabla\Mu f\in L^1(\mathbb{R}^d)\).
For the cube maximal operator \(\M\) defined above this question remains open.

For a function \(f:\mathbb{R}\rightarrow\mathbb{R}\), the variation bound for the uncentered maximal function has already been proven in \cite{MR1898539} by Tanaka
and in \cite{MR2276629} by Aldaz and P\'erez L\'azaro.
Note that in one dimension, balls and cubes are the same.
For the centered Hardy-Littlewood maximal function Kurka proved the bound in \cite{MR3310075}.
The latter proof turned out to be much more complicated.
In \cite{MR2539555} Aldaz and P\'erez L\'azaro have proven the gradient bound for the uncentered maximal operator for block decreasing functions in \(W^{1,1}(\mathbb{R}^d)\) and any dimension \(d\).
In \cite{MR3800463} Luiro has done the same for radial functions.
More endpoint results are available for related maximal operators,
for example convolution maximal operators \cite{MR3063097,carneiro2019gradient},
fractional maximal operators \cite{MR1979008,MR3624402,MR3624402,beltran2019regularity,MR3912794,weigt2021endpoint,MR3319617},
and discrete maximal operators \cite{MR3091605},
as well as maximal operators on different spaces,
such as in the metric setting \cite{MR2328816}
and on Hardy-Sobolev spaces \cite{MR3891939}.
For more background information on the regularity of maximal operators
there is a survey \cite{carneiro2019regularity} by Carneiro.
Local regularity properties of the maximal function, which are weaker than the gradient bound of the maximal operator
have also been studied \cite{MR2550181,MR2868961}.
The question whether the maximal operator is a continuous operator on the gradient level is even more difficult to answer than its boundedness
because the maximal operator is not linear.
Some progress has already been made on this question in \cite{MR2280193,MR3695894,carneiro2020sunrise,beltran2021continuity}.

This is the fourth paper in a series \cite{weigt2020variation,weigt2020variationdyadic,weigt2021endpoint} on higher dimensional variation bounds of maximal operators,
using geometric measure theory and covering arguments.
In \cite{weigt2020variation} we prove \cref{eq_goal} for the uncentered Hardy-Littlewood maximal function of characteristic functions,
in \cite{weigt2020variationdyadic} we prove it for the dyadic maximal operator for general functions,
and in \cite{weigt2021endpoint} we prove the corresponding result for the fractional maximal operator.
Here we apply tools developed in \cite{weigt2020variation,weigt2020variationdyadic}.
Note that it is not possible to extend the variation bound from characteristic to simple and then general functions, using only the sublinearity of the maximal function.
The pitfall in that strategy is that while the maximal function is sublinear, this is not true on the gradient level:
there are characteristic functions \(f_1,f_2\) such that \(\var\M(f_1+f_2)>\var\M f_1+\var\M f_2\), see \cite[Example~5.2]{weigt2020variationdyadic}.

The starting point here and in \cite{weigt2020variation,weigt2020variationdyadic} is the coarea formula,
which expresses the variation of the maximal function in terms of the boundary of the distribution set.
We observe that the distribution set of the uncentered maximal function is the union of all cubes on which the function has the corresponding average.
We divide the cubes of the distribution set of the maximal function, into two groups:
we say that those which intersect the distribution set of the function a lot have \textit{high density}, and the others have \textit{low density}.
The union of the high density cubes looks similar to the distribution set of the function,
and for characteristic functions we have already bounded its boundary in \cite{weigt2020variation} due to a result in the spirit of the isoperimetric inequality.
The motivation for that bound came from \cite[Theorem~3.1]{MR2400262} by Kinnunen, Korte, Shanmugalingam and Tuominen.
In \cite{weigt2020variationdyadic} and in this paper the high density cubes are bounded using the same argument.
This bound is even strong enough to control the low density balls for characteristic functions in the global setting in \cite{weigt2020variation}.
But in the local setting in \cite{weigt2020variation} dealing with the low density balls is more involved.
It requires a careful decomposition of the function in parallel with the low density balls of the maximal function by dyadic scales.
In that paper it also relies on the fact that the function is a characteristic function.
In \cite{weigt2020variationdyadic} we devise a strategy for dealing with the low density cubes for general functions in the dyadic cube setting.
The advantage of the dyadic setting is that the decomposition of the low density cubes and the function
are a lot more straightforward because dyadic cubes only intersect in trivial ways.
Furthermore, the argument contains a sum over side lengths of cubes which converges as a geometric sum for dyadic cubes.
In \cite{weigt2021endpoint} we bound the fractional operator,
using that it disregards small balls,
which allows for a reduction from balls to dyadic cubes so that we can apply the result from \cite{weigt2020variationdyadic}.
Non-fractional maximal operators are much more sensitive, in that we have to deal with complicated intersections of balls or cubes of any size.
In this paper we represent the low density cubes of the maximal function by a subfamily of cubes with dyadic properties, which allows to apply the key dyadic result from \cite{weigt2020variationdyadic}.
In order to make the rest of the dyadic strategy of \cite{weigt2020variationdyadic} work here,
the function is decomposed in a similar way as in the local case of \cite{weigt2020variation}.

\paragraph{Acknowledgements}
I would like to thank Panu Lahti for helpful comments on the manuscript,
and my supervisor, Juha Kinnunen for all of his support.
I am indebted to the referees for their great effort, their careful reading and their many valuable comments which have lead to numerous improvements of the manuscript.
The author has been partially supported by the Vilho, Yrjö and Kalle V\"ais\"al\"a Foundation of the Finnish Academy of Science and Letters,
and the Magnus Ehrnrooth foundation.

\section{Preliminaries and core results}
\label{sec_results}

We work in the setting of functions of bounded variation, as in Evans-Gariepy \cite{MR3409135}, Section~5. 
Let \(\Omega\subset\mathbb{R}^d\) be an open set.
A function \(f\in L^1_\loc(\Omega)\) is said to have locally bounded variation
if for every open and compactly contained set \(V\subset \Omega\) we have
\[\sup\Bigl\{\int_Vf\div\varphi:\varphi\in C^1_{\tx c}(V;\mathbb{R}^d),\ |\varphi|\leq1\Bigr\}<\infty.\]
Such a function comes with a Radon measure \(\mu\) and a \(\mu\)-measurable function \(\sigma:\Omega\rightarrow\mathbb{R}^d\) which satisfies \(|\sigma(x)|=1\) for \(\mu\)-a.e.\ \(x\in\Omega\) and such that for all \(\varphi\in C^1_{\tx c}(\Omega;\mathbb{R}^d)\) we have
\[
\int_\Omega f\div\varphi
=
\int_\Omega\varphi\sigma\intd \mu
.
\]
We define the variation of \(f\) in \(\Omega\) by
\(\var_\Omega f=\mu(\Omega)\).
If \(f\) does not have locally bounded variation or if \(f\not\in L^1_\loc(\Omega)\) then we set \(\var_\Omega f=\infty\).

For a measurable set \(E\subset\mathbb{R}^d\) denote by \(\tc E\), \(\tb E\) and \(\ti E\) the topological closure, boundary and interior of \(E\), respectively.
The measure theoretic closure, boundary and interior of $E$ are defined as
\[
\mc E
=
\Bigl\{x:\limsup_{r\rightarrow0}\f{\lm{B(x,r)\cap E}}{r^d}>0\Bigr\}
,
\qquad
\mb E
=
\mc E
\cap
\mc{\mathbb{R}^d\setminus E}
\quad\text{and}\quad
\mi E
=
\mc E
\setminus
\mb E
.
\]
Note, that \(\mc E\subset\tc E\), \(\mb E\subset\tb E\) and \(\ti E\subset\mi E\).
For a cube, its measure theoretic boundary, closure and interior agree with the respective topological objects.
We also introduce the following measure theoretic set relations:
Let \(A,B\subset\mathbb{R}^d\) be Lebesgue measurable sets.
By \(A\msubset B\) we mean \(\lm{A\setminus B}=0\), and similarly
by \(A\meq B\) we mean \(A\msubset B\) and \(B\msubset A\),
and by \(A\msubsetneq B\) we mean \(A\msubset B\) and \(B\not\msubset A\).
All these measure theoretic notions are robust against changes with Lebesgue measure zero.

For \(\Omega\subset\mathbb{R}^d\) and a function \(f:\Omega\rightarrow\mathbb{R}\) we write
\[
\{f>\lambda\}
=
\{x\in\Omega:f(x)>\lambda\}
\] 
for the superlevelset of \(f\),
and we define \(\{f\geq \lambda\}\) similarly.
Denote by \(\H^{d-1}\) the \(d-1\)-dimensional Hausdorff measure.
The following coarea formula provides an interpretation of the variation that is useful for us:

\begin{lem}[{\cite[Theorem~3.40]{MR3409135}}]
\label{lem_coareabv}
Let \(\Omega\subset\mathbb{R}^d\) be an open set and assume that \(f\in L^1_\loc(\Omega)\).
Then
\[\var_\Omega f=\int_\mathbb{R}\sm{\mb{\{f\geq\lambda\}\cap \Omega}}\intd\lambda.\]
\end{lem}

In \cite[Theorem~3.40]{MR3409135} the coarea formula is stated with the set \(\{f>\lambda\}\) in place of \(\{f\geq\lambda\}\),
but it can be proven for \(\{f\geq\lambda\}\) using the same proof.
For a set \(\Q\) of subsets of \(\mathbb{R}^d\) we denote
\[
\bigcup\Q=\bigcup_{Q\in\Q}Q
.
\]
The integral average of a function \(f\in L^1(Q)\) over a set $Q\subset\mathbb{R}^d$ with finite Lebesgue measure is denoted by
\[
f_Q
=
\f1{\lm Q}\int_Qf(x)\intd x
.
\]
For any functions \(f,g:\mathbb{R}^{n+k}\rightarrow\mathbb{R}\) we mean by \(f(a,b)\lesssim_a g(a,b)\) that for every \(a\in\mathbb{R}^n\) there exists a constant \(C>0\) such that for all \(b\in\mathbb{R}^k\) we have \(f(a,b)\leq C g(a,b)\).
For a ball \(B=B(x,r)\), a real number \(a\geq0\) and \(y\in\mathbb{R}^d\) denote \(aB=B(x,ar)\) and \(B+y=B(x+y,r)\).
For a set \(E\subset\mathbb{R}^d\) denote by \(\conv(E)\) its convex hull.

\begin{defi}
\label{defi_star}
For \(\Lambda\geq1\) we call a set \(Q\subset\mathbb{R}^d\) a \textit{\(\Lambda\)-star} if it has a ball denoted by \(\bll Q\) with radius denoted by \(\rad Q\) which satisfies \(\bll Q\subset Q\subset \Lambda\tc{\bll Q}\) and such that for any \(x\in Q\) we have \(\bcone{\bll Q}x\subset Q\).
\end{defi}

\begin{defi}
\label{defi_decomposition}
Let \(L\geq2\).
\begin{enumerate}
\item
\label{it_decomposition}
We say that a collection \(\P\) of subsets of \(\mathbb{R}^d\) is an \emph{\(L\)-decomposition} of a set \(Q\subset\mathbb{R}^d\) with finite Lebesgue measure if 
\(
Q
\msubset
\bigcup\P
,
\)
every \(P\in\P\) satisfies \(P\msubset Q\) and \(\lm Q/L\leq\lm P<\lm Q\),
and for any \(P_1,P_2\in\P\) with \(P_1\neq P_2\) we have \(\lm{P_1\cap P_2}=0\).
\item
\label{it_complete}
We say that a collection \(\Q\) of \(\Lambda\)-stars is \(L\)-\emph{complete} if it has a superset \(\P\supset\Q\) of \(\Lambda\)-stars such that every \(Q\in\P\) has an \(L\)-decomposition \(\P(Q)\subset\P\) and for all \(Q,R\in\Q\) and \(P\in\P(Q)\) with \(R\msubset P\) we have \(P\in\Q\).
\end{enumerate}
\end{defi}

\begin{rem}
\label{rem_cubeisstar}
Every bounded open convex set is a \(\Lambda\)-star for some \(\Lambda\geq1\), and every \(\Lambda\)-star is a John domain, see \cref{lem_Johnvstar}.
In particular every cube is a \(\sqrt d\)-star and the set of all cubes is \(2^d\)-complete since every cube can be decomposed into \(2^d\) cubes with half its side length.
\end{rem}

Even though a cube is convex unlike an \(L\)-complete \(\Lambda\)-star in general, our proofs barely simplify when considering only the special case of cubes.
In order to facilitate understanding, the reader is hence justified to imagine cubes when encountering \(L\)-complete \(\Lambda\)-stars in this manuscript.

Our main result is the following \lcnamecref{theo_main},
we prove it in \cref{sec_proof}:

\begin{theo}
\label{theo_main}
Let \(L\geq2\), \(\Lambda\geq1\), let \(\Q\) be a finite \(L\)-complete set of \(\Lambda\)-stars and let \(f\in L^1(\mathbb{R}^d)\) be a function with \(\var_{\bigcup\Q}f<\infty\).
For all \(\lambda\in\mathbb{R}\) denote \(\Q^\lambda=\{Q\in\Q:f_Q\geq\lambda\}\).
Then
\begin{equation}
\label{eq_main}
\int_{-\infty}^\infty
\smb{
\tb{\bigcup\Q^\lambda}
\setminus
\mc{\{f\geq\lambda\}}
}
\intd\lambda
\lesssim_{d,L,\Lambda}
\int_{-\infty}^\infty
\smb{
\mb{\{f\geq\lambda\}}
\cap
\bigcup
\{\ti Q:Q\in\Q\}
}
\intd\lambda
.
\end{equation}
\end{theo}

\begin{theo}
\label{theo_maininfinite}
Let \(L\geq2\) and \(\Lambda\geq1\).
Given an open set \(\Omega\subset\mathbb{R}^d\), let \(\Q\) be an \(L\)-complete set of \(\Lambda\)-stars \(Q\) with \(Q\subset\Omega\) and let \(f\in L^1_\loc(\Omega)\) be a function with \(\var_\Omega f<\infty\).
Then for every \(Q\in\Q\) we have \(\int_Q |f| <\infty\),
and the maximal function defined by
\[
\M_\Q f(x)
=
\max
\biggl\{
f(x)
,
\sup_{Q\in\Q,\ x\in Q}
\f1{\lm Q} \int_Q f(y)\intd y
\biggr\}
\]
belongs to \(L^{d/(d-1)}_\loc(\Omega)\) and satisfies
\[
\var_\Omega\M_\Q f
\lesssim_{d,L,\Lambda}
\var_\Omega f
.
\]
\end{theo}

\Cref{theo_maininfinite} follows from \cref{theo_main} by the coarea formula \cref{lem_coareabv} and since
\[
\{\M_\Q f\geq\lambda\}
=
\bigcup\Q^\lambda\cup\{f\geq\lambda\}
\]
if \(\Q\) is finite,
in combination with integrability and approximation arguments.
We lay out the details of this reasoning in \cref{subsec_uncountable}.

\begin{rem}
It is not clear if \cref{theo_main,theo_maininfinite} hold also without the assumption that \(\Q\) is \(L\)-complete or if we omit taking the maximum with \(f\) in the definition of \(\M_\Q f\).
The following example shows that \cref{theo_main,theo_maininfinite} fail if we drop both conditions at the same time:
Let \(f=\ind{(0,1)^d}\) and for \(N\in\mathbb{N}\) let \(\Q_N\) be the set of dyadic cubes with side length at least \(2\),
and the cubes \((n_12^{-N},(n_1+1)2^{-N})\times\ldots\times(n_d2^{-N},(n_d+1)2^{-N})\subset(0,1)^d\) where \(n_1,\ldots,n_d\) are integers such that \(n_1+\ldots+n_d\) is even.
Then the maximal operator that averages over all cubes in \(\Q_N\) has variation of the order \(2^{Nd}\cdot2^{-N(d-1)}=2^N\) which goes to infinity for \(N\rightarrow\infty\).
\end{rem}

\begin{lem}
\label{rem_localstars}
If \(\Omega\subset\mathbb{R}^d\) and \(\Q\) is an \(L\)-complete set of measurable sets then so are \(\{Q\in\Q:Q\msubset\Omega\}\) and \(\{Q\in\Q:\mc Q\subset\Omega\}\).
\end{lem}

\begin{proof}
This follows from the fact that for every \(Q\in\Q\) we have \(Q\msubset\Omega\) if and only if for all \(P\in\P(Q)\) we have \(P\msubset\Omega\),
and \(\mc Q\subset\Omega\) if and only if for all \(P\in\P(Q)\) we have \(\mc P\subset\Omega\).
\end{proof}

\begin{pro}
\label{rem_tiortc}
Let \(\Q\) be a set of \(\Lambda\)-stars and let \(\tilde\Q\) be a set of subsets of \(\mathbb{R}^d\) such that for each \(Q\in\Q\) there exists a \(\tilde Q\in\tilde\Q\), and, conversely, for each \(\tilde Q\in\tilde\Q\) there exists a \(Q\in\Q\), such that \(\tilde Q\subset\tc Q\) and \(\ti Q\msubset\tilde Q\).
Then for almost every \(x\in\Omega\) we have \(\M_{\tilde\Q}f(x)=\M_\Q f(x)\) and thus \cref{theo_maininfinite} holds with \(\tilde\Q\) in place of \(\Q\).
If \(\Q\) and \(\tilde\Q\) are finite then also \cref{theo_main} holds for \(\tilde\Q\) in place of \(\Q\) with the topological boundary and interior replaced by the measure theoretic boundary and interior in \cref{theo_main}.
If in addition \(\ti Q\subset\tilde Q\) then \cref{theo_main} holds for \(\tilde\Q\) as stated with the topological boundary and interior.

Particular examples for \(\tilde\Q\) are \(\{\tc Q:Q\in\Q\}\) and \(\{\ti Q:Q\in\Q\}\).
\end{pro}
We prove \cref{rem_tiortc} in \cref{subsec_openclosed}.

\begin{rem}
\Cref{theo_maininfinite} can fail for \(\tilde\Q\) from \cref{rem_tiortc} if we replace the condition \(\tilde Q\subset\tc Q\) by \(\tilde Q\msubset\tc Q\).
For example take a set \(A\subset[-2,-1]^d\) with
\[
\sm{\mb A}=\infty
,\qquad
\Q=\{[0,1]^d\}
,\qquad
\tilde\Q=\{[0,1]^d\cup\{x\}:x\in A\}
,\quad\tx{and}\quad
f=\ind{[0,1]^d}
.
\]
Then
\[
\var f=\sm{\tb{[0,1]^d}}<\infty
,
\]
but
\(
\M_{\tilde\Q}f=\ind{[0,1]^d\cup A}
\)
and thus
\[
\var\M_{\tilde\Q}
=
\sm{\tb{[0,1]^d}}+\sm{\mb A}
=
\infty
.
\]
\end{rem}

\begin{cor}
\label{cor_manycases}
For any open subset \(\Omega\subset\mathbb{R}^d\) and any \(f\in L^1_\loc(\Omega)\) each maximal function \(\hat\M f\) from the following list \cref{it_cubesetc,it_smallstars,it_localstars,it_closedinsteadofopen} belongs to \(L^{d/(d-1)}_\loc(\Omega)\) and satisfies
\[
\var_\Omega\hat\M f
\leq C_{d,L,\Lambda}
\var_\Omega f
.
\]
\begin{enumerate}
\item
\label{it_cubesetc}
Let \(\Q\) be the set of cubes \(Q\) contained in \(\Omega\) and let
\[
\hat\M f(x)
=
\sup_{Q\in\Q,\ x\in Q}
\f1{\lm Q} \int_Q f(y)\intd y
.
\]
Here, we may let \(\Q\) consist of cubes with any orientation, of only axes parallel cubes or of only dyadic cubes.
In particular it supersedes the main results in \cite{weigt2020variationdyadic}.
We may consider cubes to be closed or open and we may replace the condition \(Q\subset\Omega\) by \(\tc Q\subset\Omega\).

Instead of cubes \(\Q\) may consist of any open convex shapes that can be decomposed into more than one boundedly rescaled and rotated version of itself.
\item
\label{it_smallstars}
Let \(\Q\) be an \(L\)-complete set of \(\Lambda\)-stars \(Q\subset\Omega\) such that for every \(\varepsilon>0\) and almost every \(x\in\Omega\) there exists a \(Q\in\Q\) with \(x\in Q\) and \(\rad Q<\varepsilon\).
Define
\[
\hat\M f(x)
=
\sup_{Q\in\Q,\ x\in Q}
\f1{\lm Q} \int_Q f(y)\intd y
.
\]
\item
\label{it_localstars}
Let \(\Q\) be an \(L\)-complete set of \(\Lambda\)-stars \(Q\subset\mathbb{R}^d\) and let
\[
\Q_\Omega=\{Q\in\Q:Q\msubset\Omega\}
,\qquad\tx{or}\qquad
\Q_\Omega=\{Q\in\Q:\mc Q\subset\Omega\}
.
\]
Define
\[
\hat\M f(x)
=
\max
\biggl\{
f(x)
,
\sup_{Q\in\Q_\Omega,\ x\in Q}
\f1{\lm Q} \int_Q f(y)\intd y
\biggr\}
.
\]
\item
\label{it_closedinsteadofopen}
In any of the previous cases we may replace the condition \(x\in Q\) in the supremum by \(x\in\tc Q\) or \(x\in\ti Q\).
\end{enumerate}
\end{cor}

\begin{proof}[Proof of \cref{theo_goal,cor_manycases}]
In case \cref{it_smallstars} a Lebesgue density theorem holds for the stars in \(\Q\), which implies \(\hat\M f(x)\geq f(x)\) and thus \(\hat\M f(x)=\M_\Q f(x)\) for almost every \(x\in\Omega\).
Since the variation of a function does not change under modifications on sets of Lebesgue measure zero we can conclude \cref{it_smallstars} from \cref{theo_maininfinite}.
\Cref{it_localstars} follows from \cref{rem_localstars}.
\Cref{it_cubesetc} follows from \cref{rem_cubeisstar,it_smallstars,it_localstars}.
\Cref{theo_goal} is a special case of \cref{it_cubesetc}.
By \cref{rem_tiortc} the maximal function only changes on a set of measure zero when replacing the condition \(x\in Q\) in the supremum by \(x\in\tc Q\) or \(x\in\ti Q\) and we can conclude \cref{it_closedinsteadofopen}.
\end{proof}

\begin{rem}
It is usually more difficult to prove regularity results for local maximal operators than for global maximal operators,
because some arguments use that the maximal operator also takes into account certain blow-ups of balls or cubes.
In fact for the fractional maximal operator, gradient bounds which hold for the global operator fail for the local operator, see \cite[Example~4.2]{MR3319617}.
Not so here, our results hold for domains and by the same proof as for \(\mathbb{R}^d\).
\end{rem}

\begin{rem}
In \cite{weigt2020variationdyadic} we assumed that the dyadic maximal operator averages only over cubes which are compactly contained in the domain \(\Omega\) in order to ensure local integrability of the maximal function.
This condition turned out to be not necessary, provided that \(\var_\Omega f<\infty\).
\end{rem}

\begin{rem}
\label{rem_balls}
The proofs in this manuscript yield partial results also in the case of balls, i.e.\ for the classical uncentered Hardy-Littlewood maximal operator \(\Mu\).
\begin{enumerate}
\item Balls are \(1\)-stars but the collection of all balls is not \(L\)-complete for any \(L\geq2\) since a ball cannot be decomposed into a finite set of smaller balls.
That means \cref{theo_maininfinite} does not apply to the classical uncentered Hardy-Littlewood maximal operator \(\Mu\).
\item However, balls are not forbidden in the collection of sets \(\Q\) per se; one can construct a superset of the set of all balls which for some \(L\geq2\), \(\Lambda\geq1\) is an \(L\)-complete set of \(\Lambda\)-stars.
That means \cref{theo_maininfinite} applies to a maximal operator that averages over all balls if we also allow it to average over certain additional sets.
\item
\label{it_massaboveballs}
We only use \(L\)-completeness in \cref{sec_sparsemassabove}, while the arguments in \cref{subsec_preparations,sec_combination,sec_covering,sec_dyadic,sec_organizingmass,sec_approximation} work for any collection of \(\Lambda\)-stars and thus in particular for balls.
That means we provide several tools that may be directly used in further attempts to prove the endpoint variation bound \cref{eq_goal} for the uncentered Hardy-Littlewood maximal operator.
In particular, it suffices to find a suitable replacement of \cref{eq_goallambda} for balls in order to complete the proof.
\end{enumerate}
\end{rem}

\begin{rem}
The notion of a \(\Lambda\)-star may appear to be only slightly more general than that of a classical \(\Lambda\)-John domain, see \cref{defi_John,lem_Johnvstar}.
One key difference however is that a John domain may have infinite perimeter while a \(\Lambda\)-star may not.
So take a John domain \(Q\) with infinite perimeter.
Then for a function \(f\in L^1(\mathbb{R}^d)\) with \(f=0\) on \(\mathbb{R}^d\setminus Q\) and \(f_Q>0\), the maximal function
\[
\M_{\{Q\}}f(x)=\max\{f(x),f_Q\ind Q(x)\}
\qquad\tx{has}\qquad
\var\M_{\{Q\}}f\geq f_Q\sm{\mb Q}=\infty
,
\]
which means that \cref{theo_main,theo_maininfinite} do not hold if we allow \(\Lambda\)-John domains instead of \(\Lambda\)-stars.
It is unclear however if \cref{theo_maininfinite} holds if \(\Q\) is for example the set of all \(\Lambda\)-John domains.
\end{rem}

\begin{rem}
The space \(L^1_\loc(\Omega)\) is not an appropriate domain for \(\M_\Q\)
as has already been observed in \cite[footnote (2), p.~170]{MR2041705}.
For example the function \(f:x\mapsto|x|^2\) belongs to \(L^1_\loc(\Omega)\) but its classical Hardy-Littlewood maximal function is infinite everywhere.
If we strengthen the assumption to \(f\in L^1(\Omega)\), then \(\M_\Q f\) is finite almost everywhere by the weak endpoint \(p=1\) of the Hardy-Littlewood maximal function theorem.
\Cref{lem_mffinite} shows that an alternative way ensure the almost everywhere finiteness of \(\M_\Q f\) is to require \(\var_\Omega f<\infty\) in addition to \(f\in L^1_\loc(\Omega)\).
\end{rem}

\begin{rem}
\Cref{theo_goal,theo_maininfinite,cor_manycases} also extend to the maximal function of the absolute value due to
\(\var\M_\Q(|f|)\lesssim_{d,L,\Lambda}\var|f|\leq\var f\).
\end{rem}


\section{The finite case}
\label{sec_proof}

In this \lcnamecref{sec_proof} we prove \cref{theo_main}.

\subsection{Preparations}
\label{subsec_preparations}

Given a point \(x\in\mathbb{R}^d\) we denote its Euclidean length by \(|x|=\sqrt{x_1^2+\ldots+x_d^2}\).
For a set \(E\subset\mathbb{R}^d\) define \(\dist(x,E)=\inf_{y\in E}|x-y|\).

\begin{defi}
\label{defi_John}
For \(\Lambda\geq1\) a bounded open set \(Q\subset\mathbb{R}^d\) is a \emph{\(\Lambda\)-John domain} if it has a distinguished point \(x\in Q\) such that for any \(y\in Q\) there is a continuous map \(\gamma:[0,1]\to Q\) with \(\gamma(0)=x\) and \(\gamma(1)=y\) such that for any \(0\leq t\leq 1\) we have
\[
\dist(\gamma(t),\tb Q)
\geq
|y-\gamma(t)|/\Lambda
.
\]
\end{defi}

\begin{lem}
\label{lem_Johnvstar}
Let \(\Lambda\geq1\).
\begin{enumerate}
\item
\label{it_starisJohn}
Every \(\Lambda\)-star is a \(\Lambda\)-John domain.
\item
\label{it_Johnhasball}
Conversely, for every \(\Lambda\)-John domain \(Q\) there is a ball \(\bll Q\) with \(\bll Q\subset Q\subset\Lambda\tc{\bll Q}\).
\end{enumerate}
\end{lem}

\begin{proof}
For a \(\Lambda\)-star \(Q\) set \(x\) to be the center of \(\bll Q\) and for \(y\in Q\) take \(\gamma:t\mapsto(1-t)x+ty\).
This proves \cref{it_starisJohn}.
For the proof of \cref{it_Johnhasball} let \(x\) be the distinguished point of \(Q\) and let \(y_1,y_2,\ldots\in Q\) be a sequence of points with
\[
|y_n-x|\to\sup\{|x-y|:y\in Q\}=r
.
\]
Let \(\gamma_n\) be the map witnessing that \(Q\) is a John domain for \(y_n\).
Then
\[
\dist(x,\tb Q)
=
\dist(\gamma_n(0),\tb Q)
\geq
|y_n-\gamma_n(0)|/\Lambda
=
|y_n-x|/\Lambda
\to
r/\Lambda
.
\]
Therefore,
\(
B(x,r/\Lambda)
\subset
Q
\subset
\tc{
B(x,r)
}
.
\)
\end{proof}

Next, we collect a few elementary geometric properties about \(\Lambda\)-stars.

\begin{lem}
\label{lem_smallclosure}
Let \(\Lambda\geq1\) and let \(Q\) be a \(\Lambda\)-John domain.
Then
\begin{enumerate}
\item
\label{cor_starmeasurable}
\(Q\) is Lebesgue measurable
\item
\label{it_startcnearti}
and \(\mc{\ti Q}=\tc Q\).
\end{enumerate}
If in addition \(Q\) is a \(\Lambda\)-star then
\begin{enumerate}
\setcounter{enumi}{2}
\item
\label{it_interiorofboundaryconeinstar}
for any \(x\in\tc Q\) we have \(\ti\conv(\bll Q\cup\{x\})\subset Q\),
\item
\label{it_startcandtiarestars}
\(\tc Q\) and \(\ti Q\) are also \(\Lambda\)-stars
\item
\label{it_startopeneqmesopen}
and \(\ti Q=\mi Q\).
\item
\label{it_opencompletestars}
If \(\Q\) is an \(L\)-complete set of \(\Lambda\)-stars then so is \(\ti\Q\).
\end{enumerate}
\end{lem}

\begin{proof}
Since \(\ti Q\) is open and thus Lebesgue measurable the set \(\mc{\ti Q}\) is well defined and it follows from the definitions that \(\mc{\ti Q}\subset\tc Q\).
For the reverse inclusion let \(x\in\tc Q\) and let \(\varepsilon>0\).
Then there exists a \(\tilde x\in Q\) with \(|\tilde x-x|<\varepsilon\).
Let \(\gamma\) be the curve corresponding to \(\tilde x\).
Because it is continuous there is a \(0<t<1\) with \(|\gamma(t)-\tilde x|=\varepsilon\) and by definition
\(
B(\gamma(t),\varepsilon/\Lambda)
\subset
Q
.
\)
Since \(\ti Q\) is open it is measurable and thus
\[
\lm{B(x,(2+1/\Lambda)\varepsilon)\cap\ti Q}
\geq
\lm{B(\gamma(t),\varepsilon/\Lambda)}
=
(2\Lambda+1)^{-d}
\lm{B(x,(2+1/\Lambda)\varepsilon)}
\]
and we can conclude \(x\in\mc{\ti Q}\), which implies \cref{it_startcnearti}, and hence
\[
Q\setminus\ti Q\subset\tc Q\setminus\ti Q=\mc{\ti Q}\setminus\ti Q
.
\]
By the Lebesgue density theorem we have \(\lm{\mc{\ti Q}\setminus\ti Q}=0\) and thus we can conclude \(\lm{Q\setminus\ti Q}=0\) which entails \cref{cor_starmeasurable}.
Assume in addition that \(Q\) is a \(\Lambda\)-star.
It is straightforward to show that
\[
\bll Q\subset\ti Q\subset\tc Q\subset\Lambda\tc{\bll Q}
.
\]
Let \(y\in\ti\conv(\bll Q\cup\{x\})\).
Then for some \(\varepsilon>0\) we have \(B(y,2\varepsilon)\subset\conv(\bll Q\cup\{x\})\), which implies
\[
B(y,\varepsilon)\subset\conv(\bll Q\cup\{\tilde x\})\subset Q
\]
and thus \(y\in\ti Q\).
We can conclude \cref{it_interiorofboundaryconeinstar} which implies \(\conv(\bll Q\cup\{x\})\subset\tc Q\) and thus \(\tc Q\) is a \(\Lambda\)-star.

Next denote by \(c\) the center of \(\bll Q\) and let \(x\in\mi Q\) with \(x\neq c\).
Then there exists a \(y\in Q\) with
\[
\langle y-x,x-c\rangle>\sqrt{1-\Lambda^{-2}}|y-x||x-c|
\]
and thus
\begin{align*}
\langle y-x,y-c\rangle
&=
\langle y-x,x-c\rangle
+
\langle y-x,y-x\rangle
\\
&>
\bigl[
\sqrt{1-\Lambda^{-2}}|x-c|+|y-x|
\bigr]
|y-x|
\\
&\geq
\sqrt{1-\Lambda^{-2}}|y-x||y-c|
.
\end{align*}
Since the opening angle of the cone \(\conv(\bll Q\cup\{y\})\) is at least \(\arcsin(\Lambda)\) we can conclude
\[
x\in\ti\conv(\bll Q\cup\{y\})\subset\ti Q
.
\]
Since \(\ti Q\subset\mi Q\) this implies \cref{it_startopeneqmesopen}.
Moreover, we obtain
\[
\conv(\bll Q\cup\{x\})\subset\ti\conv(\bll Q\cup\{y\})\subset\ti Q
\]
from which we can deduce that \(\ti Q\) is a \(\Lambda\)-star, finishing the proof of \cref{it_startcandtiarestars}.

If \(\Q\) is an \(L\)-complete set of \(\Lambda\)-stars then by \cref{it_startcandtiarestars} the set \(\ti\Q\) consists of \(\Lambda\)-stars as well.
Moreover, by \cref{it_startopeneqmesopen} and the Lebesgue density theorem for any \(Q\in\Q\) we have \(\ti Q\meq Q\).
Since the definition of \(L\)-completeness is resistant towards changes of Lebesgue measure zero we can conclude that \(\ti\Q\) is \(L\)-complete because \(\Q\) is \(L\)-complete.
This concludes the proof of \cref{it_opencompletestars}.
\end{proof}

Instead of explicitly evoking \cref{lem_smallclosure}\cref{cor_starmeasurable}, for the rest of this manuscript we will tacitly remember that a \(\Lambda\)-star is a Lebesgue measurable set when needed.

By \cite[Theorem~107]{lecturenoteshajlasz} every John domain satisfies a Poincar\'e inequality.
It is our most important building block and used repeatedly in this manuscript.

\begin{theo}[{\cite[Theorem~107]{lecturenoteshajlasz}}]
\label{theo_poincare}
Let \(\Lambda\geq1\), let \(Q\) be a \(\Lambda\)-John domain and let \(f\in L^1(Q)\) with locally bounded variation.
Then
\[
\Bigl(
\int_Q
|f-f_Q|^{d/(d-1)}
\Bigr)^{(d-1)/d}
\lesssim_{d,\Lambda}
\var_Q f
.
\]
\end{theo}

\begin{rem}
\label{rem_medianpoincare}
By Hölder's inequality applied to \cref{theo_poincare} we have
\begin{equation}
\label{eq_oneonepoincare}
\int_Q
|f-f_Q|
\lesssim_{d,\Lambda}
\lm Q^{\f1d}
\var_Q f
.
\end{equation}
Let \(f\) be finite almost everywhere.
It is an exercise to show that with the median
\[
m
=
\inf\{\lambda\in\mathbb{R}:\lm{Q\cap\{f>\lambda\}}\leq\lm Q/2\}
\]
we have
\[
\f12
\int_Q
|f-f_Q|
\leq
\int_Q
|f-m|
=
\inf_{c\in\mathbb{R}}
\int_Q
|f-c|
\leq
\int_Q
|f-f_Q|
\]
with all values in this chain of inequalities understood to be infinite if \(f\not\in L^1(Q)\).
Let \(f\in L^1_\loc(Q)\) with \(\var_Qf<\infty\).
Then \(f\) is finite almost everywhere, and thus its median is finite too.
Approximating \(f\) by its truncations
\[
f_n=\max\{\min\{f,m+n\},m-n\}\in L^1(Q)
\]
which for \(n\) large enough have the same median as \(f\) we can conclude
\begin{align*}
\int|f-m|
&=
\lim_{n\rightarrow\infty}
\int|f_n-m|
\leq2
\lim_{n\rightarrow\infty}
\int|f_n-(f_n)_Q|
\lesssim_{d,\Lambda}
\lim_{n\rightarrow\infty}
\lm Q^{\f1d}
\var_Qf_n
\\
&\leq
\lm Q^{\f1d}
\var_Q f
<\infty
,
\end{align*}
i.e.\ \cref{eq_oneonepoincare} holds also with \(f_Q\) replaced by \(m\).
Moreover, we can conclude \(f\in L^1(Q)\) and thus \cref{theo_poincare,eq_oneonepoincare} still hold if we weaken the assumption \(f\in L^1(Q)\) to \(f\in L^1_\loc(Q)\).
\end{rem}

The relative isoperimetric inequality holds as a corollary of \cref{theo_poincare}.

\begin{cor}
\label{lem_isoperimetric}
Let \(\Lambda\geq1\), let \(Q\) be a \(\Lambda\)-John domain and let \(E\subset\mathbb{R}^d\) be Lebesgue measurable.
Then
\[
\min\{\lm{Q\cap E},\lm{Q\setminus E}\}^{d-1}
\lesssim_{d,\Lambda}
\sm{Q\cap\mb E}^d
.
\]
\end{cor}

\begin{proof}
Since \(\mb E=\mb{(\mathbb{R}^d\setminus E)}\) it suffices to consider the case \(\lm{Q\cap E}\leq\lm Q/2\).
Then \((\ind E)_Q\leq1/2\) and thus \(|\ind E-(\ind E)_Q|\geq 1/2\) on \(Q\cap E\) and thus
\[
\int_Q
|\ind E-(\ind E)_Q|^{d/(d-1)}
\geq
\int_{Q\cap E}
(1/2)^{d/(d-1)}
=
(1/2)^{d/(d-1)}
\lm{Q\cap E}
\]
and by the coarea formula \cref{lem_coareabv} we have
\[
\var_Q\ind E
=
\int_0^1
\sm{Q\cap\mb E}
=
\sm{Q\cap\mb E}
.
\]
By \cref{theo_poincare} we can conclude
\[
\lm{Q\cap E}^{d-1}
\leq2^d
\Bigl(
\int_Q
|\ind E-(\ind E)_Q|^{d/(d-1)}
\Bigr)^{d-1}
\lesssim_{d,\Lambda}
(\var_Q\ind E)^d
=
\sm{Q\cap\mb E}^d
.
\qedhere
\]
\end{proof}

\begin{cor}
\label{lem_isoperimetricboundedaway}
For \(\Lambda\geq1\) and \(\lambda>0\), let \(Q\) be a \(\Lambda\)-John domain and let \(E\subset\mathbb{R}^d\) with
\[
\lambda
\leq
\f{
\lm{Q\cap E}
}{
\lm Q
}
\leq
1-\lambda
.
\]
Then
\[
\sm{Q\cap\mb E}
\gtrsim_{d,\Lambda}
\lambda^{(d-1)/d}
\lm Q^{(d-1)/d}
.
\]
\end{cor}

\begin{proof}
By \cref{lem_isoperimetric} we have
\[
\lambda
\lm Q
\leq
\min\{\lm{Q\cap E},\lm{Q\setminus E}\}
\lesssim_{d,\Lambda}
\sm{Q\cap\mb E}^{d/(d-1)}
.
\qedhere
\]
\end{proof}

\begin{defi}
\label{defi_qli}
Let \(L\geq2\), \(\Lambda\geq1\) and let \(\Q\) be a finite set of \(\Lambda\)-stars and let \(f\in L^1(\mathbb{R}^d)\).
For \(\lambda\in\mathbb{R}\) set \(\Q^\lambda=\{Q\in\Q:f_Q\geq\lambda\}\).
We consider the partition
\[
\Q^\lambda
=
\Q_0^\lambda\cup\Q_1^\lambda\cup\Q_2^\lambda
,
\]
where the set \(\Q_0^\lambda\) consists of all stars \(Q\in\Q^\lambda\) with
\[
\lm{Q\cap\{f\geq\lambda\}}\geq(4L)^{-1}\lm Q
,
\]
the set \(\Q_1^\lambda\) consists of all stars \(Q\in\Q^\lambda\setminus\Q_0^\lambda\) with
\[
\lmb{Q\cap\bigcup\Q_0^\lambda}\geq2^{-d-1}\Lambda^{-d}\lm Q
,
\]
and \(\Q_2^\lambda=\Q^\lambda\setminus(\Q_0^\lambda\cup\Q_1^\lambda)\).
\end{defi}

\begin{lem}[{\cite[Lemma~1.6]{weigt2020variation}}]
\label{lem_boundaryofunion}
Let \(A,B\subset\mathbb{R}^d\) be measurable. 
Then
\[
\tbe{A\cup B}\subset(\tb A\setminus\tc B)\cup\tb B.
\]
In addition, the measure theoretic variant
\(
\mbe{A\cup B}\subset(\mb A\setminus\mc B)\cup\mb B
\)
and the mixed variant
\(
\tbe{A\cup B}\subset(\tb A\setminus\mc B)\cup\tb B
\)
hold.
\end{lem}

\begin{proof}
We have
\begin{align*}
\tbe{A\cup B}
&=
\tc{A\cup B}\cap\tc{\mathbb{R}^d\setminus(A\cup B)}
\\
&\subset
(\tc A\cup\tc B)\cap\tc{\mathbb{R}^d\setminus A}\cap\tc{\mathbb{R}^d\setminus B}
\\
&=
(\tc A\cap\tc{\mathbb{R}^d\setminus A}\cap\tc{\mathbb{R}^d\setminus B})
\cup
(\tc B\cap\tc{\mathbb{R}^d\setminus A}\cap\tc{\mathbb{R}^d\setminus B})
\\
&=
(\tb A\cap\tc{\mathbb{R}^d\setminus B})
\cup
(\tb B\cap\tc{\mathbb{R}^d\setminus A})
\\
&=
(\tb A\cap\tc{\mathbb{R}^d\setminus B}\cap\tc B)
\cup
(\tb A\cap\tc{\mathbb{R}^d\setminus B}\setminus\tc B)
\cup
(\tb B\cap\tc{\mathbb{R}^d\setminus A})
\\
&\subset
(\tb A\cap\tb B)
\cup
(\tb A\setminus\tc B)
\cup
\tb B
\\
&=
(\tb A\setminus\tc B)
\cup
\tb B
.
\end{align*}
The exact same proof with the topological boundary and closure replaced by the measure theoretic boundary and closure also proves the measure theoretic variant of the inclusion.
Since \(\mc E\subset\tc E\), the mixed variant directly follows from the topological variant.
\end{proof}

Using \cref{lem_boundaryofunion} we obtain
\begin{equation}
\label{eq_decomposition}
\begin{aligned}
\smb{\tb{
\bigcup\Q^\lambda
}
\setminus
\mc{
\{f\geq\lambda\}}
}
&\leq
\smb{
\tbb{
\bigcup\Q_0^\lambda
\cup
\bigcup\Q_1^\lambda
}
\setminus
\mc{\{f\geq\lambda\}}
}
\\
&\qquad+
\smb{
\tb{
\bigcup\Q_2^\lambda
}
\setminus
\tc{
\bigcup
(\Q_0^\lambda\cup\Q_1^\lambda)
}
}
.
\end{aligned}
\end{equation}
We bound the first summand in \cref{sec_covering}.
In \cref{sec_dyadic,sec_sparsemassabove,sec_organizingmass} we collect the necessary ingredients to deal with the second summand.
In \cref{sec_combination} we combine these results to a proof of \cref{theo_main}.
The results in \cref{sec_covering,sec_sparsemassabove} are generalizations of results in \cite{weigt2020variation,weigt2020variationdyadic}.

\subsection{Stars with large intersection}
\label{sec_covering}

The main results in this \lcnamecref{sec_covering} are \cref{pro_levelsets_finite_g,cla_largeboundaryinball}, two bounds on the perimeter of a union of \(\Lambda\)-stars.
They generalize \cite[Lemma~4.1 and Proposition~4.3]{weigt2020variation} from balls to stars.

For \(c\in(c_1,\ldots,c_d)\in\mathbb{R}^d\) denote \(c_{<d}=(c_1,\ldots,c_{d-1})\in\mathbb{R}^{d-1}\) and denote by \(B_{d-1}(x,r)\) a ball in \(\mathbb{R}^{d-1}\).

\begin{lem}
\label{lem_conelipschitz}
Let \(\Lambda\geq1\), \(r>0\), \(c\in\mathbb{R}^d\) and let \(U\) be a convex set with \(B(c,r)\subset U\subset\tc{B(c,\Lambda r)}\).
Then there is a \(2\Lambda\)-Lipschitz function \(f:B_{d-1}(c_{<d},r/2)\rightarrow(c_d,\infty)\) such that for any \(u\in B_{d-1}(c_{<d},r/2)\) and \(s\geq c_d\) with \(f(u)\neq s\) we have
\(
(u_1,\ldots,u_{d-1},s)\in U
\)
if and only if
\(
f(u)<s
.
\)
\end{lem}

\begin{figure}
\centering
\includegraphics{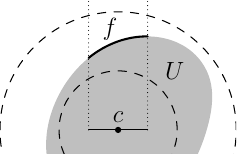}
\caption{The \(2\Lambda\)-Lipschitz function \(f\) in \cref{lem_conelipschitz} traces a segment of the boundary of \(U\).}
\end{figure}

\begin{proof}
By translation and scaling it suffices to consider the case \(c=0\) and \(r=1\).
For \(u\in\mathbb{R}^{d-1}\) and \(s\in\mathbb{R}\) we denote \((u;s)=(u_1,\ldots,u_{d-1},s)\in\mathbb{R}^d\).
For \(u\in B_{d-1}(0,1)\) define
\[
f(u)
=
\sup\{s\in\mathbb{R}:(u;s)\in U\}
>0
.
\]
Since any vertical line intersected with a convex set is an interval and
since \(B_{d-1}(0,1)\times\{0\}\subset U\) we can conclude that for any \(u\in B_{d-1}(0,1)\) and \(s\geq0\) with \(f(u)\neq s\) we have
\((u;s)\in U\) if and only if \(f(u)<s\).
It remains to show that \(f\) is \(2\Lambda\)-Lipschitz on \(B_{d-1}(0,1/2)\).
To that end let \(u,v\in B_{d-1}(0,1/2)\) with \(f(u)<f(v)\).
For \(x\in\mathbb{R}^d\) denote by
\(
U(x)
\)
the interior of the convex hull of
\(
B(0,1)\cup\{x\}
.
\)
Then for every \(0<\varepsilon<f(v)\) by convexity we have \(U(v;f(v)-\varepsilon)\subset U\) and thus
\[
U(v;f(v))
=
\bigcup_{0<\varepsilon<f(v)}
U(v;f(v)-\varepsilon)
\subset
U
.
\]
That means
\[
f(u)
\geq
\sup\{s:(u;s)\in U(v;f(v))\}
.
\]
Denote
\[
x
=
\f{
f(u)
}{
f(v)
}
(v;f(v))
.
\]
Then
\[
B\Bigl(x;
\f{
|x-(v;f(v))|
}{
|(v;f(v))|}
\Bigr)
\subset
U(v;f(v))
\subset
U
\]
which implies
\[
|u-x_{<d}|
=
|(u;f(u))-x|
\geq
\f{
|x-(v;f(v))|
}{
|(v;f(v))|
}
=
\f{
f(v)-f(u)
}{
f(v)
}
.
\]
Using
\[
|x_{<d}-v|
=
\f{
f(v)-f(u)
}{
f(v)
}
|v|
\]
we can conclude
\[
|u-v|
\geq
|u-x_{<d}|
-
|x_{<d}-v|
\geq
\f{
f(v)-f(u)
}{
f(v)
}
(1-|v|)
\geq
\f{
f(v)-f(u)
}{
2\Lambda
}
.
\]
This means \(f\) is \(2\Lambda\)-Lipschitz on \(B_{d-1}(0,1/2)\), finishing the proof.
\end{proof}

\begin{pro}[{\cite[Lemma~4.1]{weigt2020variation} for stars}]
\label{cla_largeboundaryinball}
Let \(K>0,\Lambda\geq1\), let \(B\) be a ball and let \(\Q\) be a set of \(\Lambda\)-stars \(Q\) with \(\rad Q\geq K\rad B\). 
Then
\[
\smb{B\cap\tb{\bigcup\Q}}
\lesssim_d
(K^{-1}+1)\Lambda^d\sm{\tb B}
.
\]
\end{pro}

\begin{proof}
Note, that \(\tb{\bigcup\Q}\) is a closed set and thus \(\H^{d-1}\)-measurable.
Recall that for \(x=(x_1,\ldots,x_d)\in\mathbb{R}^d\) we denote \(x_{<d}=(x_1,\ldots,x_{d-1})\).
By translation and scaling it suffices to consider the case \(B=B(0,1)\), and we
first consider the case \(K\geq 5\).
It suffices to consider those \(Q\in\Q\) which intersect \(B(0,1)\) because all other \(Q\) do not contribute to \(B(0,1)\cap\tb{\bigcup\Q}\).
Since \(Q\subset\Lambda\tc{\bll Q}\) this means the center \(c\) of \(\bll Q\) satisfies \(|c|<\Lambda\rad Q+1\).
For a unit vector \(e\in\Sph^{d-1}\) denote by \(\Q_e\) the set of those \(Q\in\Q\) whose centers \(c\) have an angle with \(e\) of at most \(1/(4\Lambda)\).
First consider the case \(e=(0,\ldots,0,-1)\).
For every \(Q\in\Q_e\) we have
\[
\f{
|c_{<d}|
}{
|c|
}
\leq
\sin\Bigl(
\f1{4\Lambda}
\Bigr)
\leq
\f1{4\Lambda}
\]
and therefore
\[
\f{
\rad Q
}2
-
|c_{<d}|
\geq
\f{
\rad Q
}2
-
\f{|c|}{4\Lambda}
>
\f{
\rad Q
}4
-
\f1{4\Lambda}
\geq
\f54
-
\f1{4\Lambda}
\geq 1
,
\]
which means \(B_{d-1}(0,1)\subset B_{d-1}(c_{<d},\rad Q/2)\).
For \(x\in B(0,1)\) apply \cref{lem_conelipschitz} with \(U=\conv(\bll Q\cup\{x\})\) and \(B(c,r)=\bll Q\) and denote by \(f_{Q,x}:B_{d-1}(c_{<d},\rad Q/2)\rightarrow(c_d,\infty)\) the resulting \(2\Lambda\)-Lipschitz function.
Unless \(\Q_e\) is empty, the assignment
\[
f_e(u)
\seq
\sup_{Q\in\Q_e,\ x\in Q,\ x_d>-1}
f_{Q,x}(u)
\]
defines a function \(f_e:B_{d-1}(0,1)\rightarrow\mathbb{R}\) which as the supremum of \(2\Lambda\)-Lipschitz functions is itself \(2\Lambda\)-Lipschitz.
It suffices to consider the case that for every \(Q\in\Q\) the center \(c\) of \(\bll Q\) satisfies \(|c|>2\) because since \(K\geq5\) otherwise \(B(0,1)\subset\bll Q\subset Q\) and thus \(\tb{\bigcup\Q}\cap B(0,1)=\emptyset\).
Thus we have \(c_d\leq-|c|\cos(1/(4\Lambda))<-1\).

Let \(x\in B_{d-1}(0,1)\times(-1,\infty)\) with \(x_d<f_e(x_{<d})\).
Then there exist \(Q\in\Q_e\) with \(x_d>-1>c_d\) and \(y\in Q\) with \(y_d>-1\) and \(x_d<f_{Q,y}(x_{<d})\).
By \cref{lem_conelipschitz} we can conclude
\[
x\in \conv(\bll Q\cup\{y\})\subset Q\subset\bigcup\Q_e
.
\]
Conversely, let \(x\in B_{d-1}(0,1)\times(-1,\infty)\) with \(x_d>f_e(x_{<d})\) and let \(Q\in\Q_e\).
Then \(x_d>-1>c_d\) and since \(x\in\conv(\bll Q\cup\{x\})\) we have by \cref{lem_conelipschitz} that \(x_d\leq f_{Q,x}(x_{<d})\).
By definition of \(f_e\) this requires \(x\not\in Q\).
That means \(x\not\in\bigcup\Q_e\).
Since \(f_e\) is Lipschitz on \(B_{d-1}(0,1)\) and \(B_{d-1}(0,1)\times(-1,\infty)\) is open, we can conclude that for all \(x\in B_{d-1}\times(-1,\infty)\) we have \(f_e(x_{<d})=x_d\) if and only if \(x\in\tb{\bigcup\Q_e}\).
Because the graph of a Lipschitz function satisfies a surface measure bound we can conclude
\begin{align*}
\smb{B(0,1)\cap\tb{\bigcup\Q_e}}
&\leq
\smb{
	\bigl[B_{d-1}(0,1)\times(-1,\infty)\bigr]
	\cap
	\tb{\bigcup\Q_e}
}
\\
&=
\sm{\{(u_1,\ldots,u_{d-1},s):u\in B_{d-1}(0,1),\ f_e(u)=s\}}
\\
&\leq
\sm{B_{d-1}(0,1)}
\sqrt{1+(2\Lambda)^2}
\\
&\lesssim_d
\Lambda
.
\end{align*}
By rotation the previous estimate holds for any unit vector \(e\in\Sph^{d-1}\).
There is a grid \(G\subset\Sph^{d-1}\) with cardinality \(\# G\lesssim_d\Lambda^{d-1}\) such that for every \(\nu\in\Sph^{d-1}\) there is an \(e\in G\) which has an angle less than \(1/(4\Lambda)\) with \(\nu\).
We can conclude
\begin{align*}
\smb{B(0,1)\cap\tb{\bigcup\Q}}
&=
\smb{B(0,1)\cap\tb{\bigcup_{e\in G}\bigcup\Q_e}}
\leq
\sum_{e\in G}
\smb{B(0,1)\cap\tb{\bigcup\Q_e}}
\\
&\lesssim_d
\# G
\Lambda
\lesssim_d
\Lambda^d
,
\end{align*}
finishing the proof for \(K\geq 5\).

For \(K<5\) take a covering \(\B\) of \(B\) by a dimensional constant times \((5/K)^d\) many balls \(B'\) with \(r(B')=(K/5)r(B)\).
That means for any \(Q\in\Q\) we have \(\rad Q\geq Kr(B)=5r(B')\)
and by the above case we can conclude
\begin{align*}
\smb{\tb{\bigcup\Q}\cap B}
&\leq
\sum_{B'\in\B}
\smb{\tb{\bigcup\Q}\cap B'}
\lesssim_d
\sum_{B'\in\B}
\Lambda^d\sm{\tb{B'}}
\\
&\lesssim_d
K^{-d}
\Lambda^d
K^{d-1}
\sm{\tb B}
=
K^{-1}
\Lambda^d
\sm{\tb B}
.
\qedhere
\end{align*}
\end{proof}

\begin{rem}
In \cite{weigt2020variation} the original version of \cref{cla_largeboundaryinball}, \cite[Lemma~4.1]{weigt2020variation}, has the unnecessarily large factor \(K^{-d}\) instead of \(K^{-1}\) due to an oversight.
\end{rem}

For \(t\geq0\) denote
\[
Q(t)
=
\bigl\{
y\in Q:\dist(y,\mathbb{R}^d\setminus Q)> t
\bigr\}
.
\]

\begin{lem}
\label{lem_volumeinnerstar}
Let \(Q\) be a \(\Lambda\)-star.
Then for any \(t\geq0\) we have
\[
\lm{Q\setminus Q(t)}
\leq
\f{
dt
}{
\rad Q
}
\lm Q
.
\]
\end{lem}
\begin{proof}
By translation consider the case that \(\bll Q\) is centered in the origin.
Let \(e\in\Sph^{d-1}\) and let \(s\geq0\) with \(se\in Q\).
Then for all \(0\leq u\leq1\) we have
\[
B((1-u)se,u\rad Q)\subset\bcone{\bll Q}{se}\subset Q
\]
and thus if \(u> t/\rad Q\) then
\[
(1-u)se\in B((1-u)se,u\rad Q-t)\subset Q(t)
.
\]
With \(s_e=\sup\{s:se\in Q\}\) this means
\[
\{se:0\leq s<(1-t/\rad Q)s_e\}\subset Q(t)
.
\]
By Fubini we can conclude
\begin{align*}
\lm{Q(t)}
&=
\int_{\Sph^{d-1}}
\int_{\{s:se\in Q(t)\}}
s^{d-1}
\intd s
\intd\nu
\\
&\geq
\int_{\Sph^{d-1}}
\int_0^{(1-t/\rad Q)s_e}
s^{d-1}
\intd s
\intd\nu
\\
&=
(1-t/\rad Q)^d
\int_{\Sph^{d-1}}
\int_0^{s_e}
s^{d-1}
\intd s
\intd\nu
\\
&=
(1-t/\rad Q)^d
\lm Q
\\
&\geq
(1-dt/\rad Q)
\lm Q
.
\qedhere
\end{align*}
\end{proof}

\begin{lem}
\label{lem_starinnerstar}
Let \(Q\) be a \(\Lambda\)-star and \(t<\rad Q\).
Then \(Q(t)\) is a \(\f{\rad Q}{\rad Q-t}\Lambda\)-star.
\end{lem}

\begin{proof}
Define \(B(c,r)=\bll Q\) and let \(x\in Q(t)\).
Then there exists an \(\varepsilon>0\) such that \(B(x,t+\varepsilon)\subset Q\) and thus
\[
A
\seq
\bigcup_{y\in B(x,t+\varepsilon)}
\bcone{B(c,r)}y
\subset
Q
.
\]
Since \(B(x,t+\varepsilon)\) is convex we have \(A=\conv(B(c,r)\cup B(x,t+\varepsilon))\) which implies \(B(c,r-t)\subset A(t)\), \(x\in A(t)\) and that \(A\) and hence \(A(t)\) are convex.
We can conclude
\[
\bconeb{\f{\rad Q-t}{\rad Q}\bll Q}x
=
\cone cx{r-t}
\subset
A(t)
\subset
Q(t)
,
\]
finishing the proof.
\end{proof}

\begin{lem}[{\cite[Lemma~2.3]{weigt2020variation} for stars}]
\label{lem_perimeterinconeinterior}
Let \(0\leq\lambda\leq1\), \(\Lambda\geq1\), and let \(Q\subset\mathbb{R}^d\) be a \(\Lambda\)-star.
Let \(E\subset\mathbb{R}^d\) with \(\lm{E\cap Q}\geq\lambda\lm Q\) and let \(x\in\tb Q\setminus\mc E\).
Then there exists a \(0<t\leq\diam(Q)\) such that
\[
\sm{
\mb E\cap Q(\lambda t/(4d\Lambda))\cap B(x,t)
}
\gtrsim_{d,\Lambda}
\lambda^{\f{d-1}d}
\sm{\tb{
B(x,t)
}}
.
\]
\end{lem}

\begin{figure}
\centering
\includegraphics{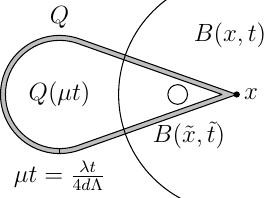}
\caption{The ball \(B(\tilde x,\tilde t)\) away from the boundary of \(Q\) in \cref{lem_perimeterinconeinterior}.}
\end{figure}

\begin{proof}
Denote \(B(c,r)=\bll Q\) and abbreviate \(\mu=\lambda/(4d\Lambda)\) and note that since \(x\in\tb Q\) we have \(|x-c|\geq r\).
For \(0<t\leq2|x-c|\) define
\[
u_t=t/(2|x-c|)
,\qquad
r_t
=
u_tr/2
,\qquad
y_t=u_tc+(1-u_t)x
.
\]
Then \(0<u_t\leq 1\) and by \cref{lem_smallclosure}\cref{it_interiorofboundaryconeinstar} we have
\[
B(y_t,2r_t)
\subset
\ti\conv(B(c,r)\cup\{x\})
\subset
Q
.
\]
Since
\[
2r_t-r_t
=
r_t
=
\f{rt}{4|x-c|}
\geq
\f t{4\Lambda}
\geq
\mu t
\]
we can conclude
\(
B(y_t,r_t)
\subset
Q(\mu t)
.
\)
Moreover,
\[
|y_t-x|
+
r_t
=
u_t|x-c|
+
\f{
u_tr
}2
=
t\Bigl(
\f12
+
\f r{
4|x-c|
}
\Bigr)
\leq
t
\]
which means
\(
B(y_t,r_t)
\subset
B(x,t)
.
\)
Since \(x\not\in\mc E\) we can conclude
\[
\f{
\lm{E\cap B(y_t,r_t)}
}{
\lm{B(y_t,r_t)}
}
\leq
(4\Lambda)^d
\f{
\lm{E\cap B(x,t)}
}{
\lm{B(x,t)}
}
\rightarrow0
\]
for \(t\rightarrow0\).

We first consider the case
\[
\lm{E\cap B(c,r/2)}
\geq
\lm{B(c,r/2)}
/2
.
\]
Since \(y_{2|x-c|}=c\) and \(r_{2|x-c|}=r/2\)
and since
\(t\mapsto\lm{B(y_t,r_t)}\)
and
\(t\mapsto\lm{E\cap B(y_t,r_t)}\)
are continuous maps this means that there exists a \(0<t\leq2|x-c|\) such that
\[
\lm{E\cap B(y_t,r_t)}
=
\lm{B(y_t,r_t)}
/2
.
\]
Note that a ball \(B\) is a \(1\)-John domain with
\[
\lm B^{d-1}
\sim_d
r(B)^{d(d-1)}
\sim_d
\sm{\tb B}^d
.
\]
Thus by \cref{lem_isoperimetricboundedaway} we can conclude
\[
\sm{\mb E\cap Q(\mu t)\cap B(x,t)}
\geq
\sm{\mb E\cap B(y_t,r_t)}
\gtrsim_d
\sm{\tb{B(y_t,r_t)}}
\geq
\f{
\sm{\tb{B(x,t)}}
}{
(4\Lambda)^{d-1}
}
\]
finishing the proof in this case.

It remains to consider the case
\[
\lm{B(c,r/2)\setminus E}
>
\lm{B(c,r/2)}
/2
.
\]
Since
\[
\mu\diam(Q)
\leq
2\Lambda r\mu
=
\f{
\lambda r
}{
2d
}
\leq
\f r2
\]
we have
\(
B(c,r/2)
\subset
Q(\mu\diam(Q))
\)
and thus
\begin{equation}
\label{eq_QmudiamQEleq}
\lm{Q(\mu\diam(Q))\setminus E}
\geq
\lm{B(c,r/2)\setminus E}
>
\f{
\lm{B(c,r/2)}
}
2
=
\f{
\lm{B(c,\Lambda r)}
}{
2^{d+1}\Lambda^d
}
\geq
\f{
\lm{Q(\mu\diam(Q))}
}{
2^{d+1}\Lambda^d
}
\end{equation}
By \cref{lem_volumeinnerstar} we have
\[
\lm{Q\setminus Q(\mu\diam(Q))}
\leq
\f{
d\lambda\diam(Q)
}{
4d\Lambda r
}
\lm Q
\leq
\f\lambda2
\lm Q
\]
and thus
\begin{equation}
\label{eq_QmudiamQEgeq}
\lm{Q(\mu\diam(Q))\cap E}
\geq
\lambda\lm Q
-
\lm{Q\setminus Q(\mu\diam(Q))}
\geq
\f\lambda2
\lm Q
\geq
\f\lambda2
\lm{Q(\mu\diam(Q))}
.
\end{equation}
Since \(\mu\diam(Q)\leq\rad Q/d\) by \cref{lem_Johnvstar}\cref{it_starisJohn} and \cref{lem_starinnerstar} the set \(Q(\mu\diam(Q))\) is a \(d\Lambda/(d-1)\)-John domain and thus by \cref{eq_QmudiamQEleq,eq_QmudiamQEgeq,lem_isoperimetricboundedaway} we can conclude
\begin{align*}
\sm{B(x,\diam(Q))\cap Q(\mu\diam(Q))\cap\mb E}
&\geq
\sm{Q(\mu\diam(Q))\cap\mb E}
\\
&\gtrsim_{d,\Lambda}
\min\{2^{-(d+1)}\Lambda^{-d},\lambda/2\}^{\f{d-1}d}
\lm{Q(\mu\diam(Q))}^{\f{d-1}d}
\\
&\gtrsim_{d,\Lambda}
\lambda^{\f{d-1}d}
\sm{\tb{B(x,\diam(Q))}}
.
\qedhere
\end{align*}
\end{proof}

For a set \(\Q\) of subsets of \(\mathbb{R}^n\) abbreviate \(\ti\Q=\{\ti Q:Q\in\Q\}\).

\begin{pro}[{\cite[Proposition~4.3]{weigt2020variation} for stars}]
\label{pro_levelsets_finite_g}
Let \(\lambda\in(0,1)\).
Let \(E\subset\mathbb{R}^d\) be a set of locally finite perimeter and let \(\Q\) be a finite set of \(\Lambda\)-stars such that for each \(Q\in\Q\) we have \(\lm{E\cap Q}\geq\lambda\lm Q\).
Then
\[
\smb{
\tb{\bigcup\Q}
\setminus\mc E
}
\lesssim_{d,\Lambda}
(1-\log\lambda)
\lambda^{-\f{d-1}d}
\smb{
\mb E \cap
\bigcup\ti\Q
}
.
\]
\end{pro}

\begin{proof}
Since \(\Q\) is finite for every \(x\in\tb{\bigcup\Q}\setminus\mc E\) there is a \(Q\in\Q\) with \(x\in\tb Q\setminus\mc E\).
Let \(\B\) be the set of all balls from \cref{lem_perimeterinconeinterior} applied to every such \(x\) and \(Q\).
For each \(n\in\mathbb{Z}\) apply the Vitali covering lemma to \(\{B\in\B:2^n<r(B)\leq2^{n+1}\}\) and denote by \(\B_n\) the resulting set, i.e.\ the balls in \(\B_n\) are disjoint and for each \(B\in\B\) with \(2^n<r(B)\leq 2^{n+1}\) we have \(B\subset\bigcup\{5C:C\in\B_n\}\).
Let
\[
\Q_n=\{Q\in\Q:2^{n-1}<\diam(Q)\leq2^n\}
\]
and \(\Q_{\geq n}=\bigcup_{k\geq n}\Q_n\).
Then
\begin{equation}
\label{eq_applyperimetercover}
\tb{\bigcup\Q}
\cap
\tb{\bigcup\Q_n}
\setminus\mc E
\subset
\bigcup_{k\leq n}
\bigcup_{C\in\B_k}
5C
.
\end{equation}
Using \cref{eq_applyperimetercover} and that \(\Q\) is finite we obtain
\begin{align*}
\tb{\bigcup\Q}
\setminus\mc E
&=
\bigcup_{n\in\mathbb{Z}}
\tb{\bigcup\Q}
\cap
\tb{\bigcup\Q_n}
\setminus\mc E
\\
&=
\bigcup_{n\in\mathbb{Z}}
\bigcup_{k\leq n}
\bigcup_{C\in\B_k}
5C
\cap
\tb{\bigcup\Q}
\cap
\tb{\bigcup\Q_n}
\setminus\mc E
\\
&=
\bigcup_{k\in\mathbb{Z}}
\bigcup_{C\in\B_k}
\bigcup_{n\geq k}
5C
\cap
\tb{\bigcup\Q}
\cap
\tb{\bigcup\Q_n}
\setminus\mc E
\\
&=
\bigcup_{k\in\mathbb{Z}}
\bigcup_{C\in\B_k}
5C
\cap
\tb{\bigcup\Q}
\cap
\tb{\bigcup\Q_{\geq k}}
\setminus\mc E
\\
&\subset
\bigcup_{k\in\mathbb{Z}}
\bigcup_{C\in\B_k}
5C
\cap
\tb{\bigcup\Q_{\geq k}}
.
\end{align*}
By \cref{cla_largeboundaryinball} we can conclude
\begin{equation}
\label{eq_boundarybyballs}
\smb{
\tb{\bigcup\Q}
\setminus\mc E
}
\leq
\sum_{k\in\mathbb{Z}}
\sum_{C\in\B_k}
\sm{
5C\cap
\tb{\bigcup\Q_{\geq k}}
}
\lesssim_{d,\Lambda}
\sum_{k\in\mathbb{Z}}
\sum_{C\in\B_k}
\sm{\tb C}
\end{equation}
By \cref{lem_perimeterinconeinterior} and since every ball \(C\in\B_k\) satisfies \(r(C)>2^k\) we have
\begin{align}
\label{eq_ballfromboundarycover}
\sm{\tb C}
\lesssim_{d,\Lambda}\lambda^{-\f{d-1}d}
\max_{Q\in\Q}
\sm{
\mb E\cap Q(\lambda 2^k/(4d\Lambda))\cap C
}
\end{align}
Abbreviate
\[
A_k
=
\bigcup\{Q(\lambda 2^k/(4d\Lambda)):Q\in\Q\}
\setminus
\bigl(\bigcup\Q\bigr)(2^{k+1})
.
\]
Since the balls \(C\in\B_k\) are pairwise disjoint, centered on the boundary of \(\bigcup\Q\) and satisfy \(r(C)\leq2^{k+1}\) we have
\begin{equation}
\label{eq_ballssamescale}
\sum_{C\in\B_k}
\max_{Q\in\Q}
\sm{
\mb E\cap Q(\lambda 2^k/(4d\Lambda))\cap C
}
\leq
\sm{\mb E\cap A_k}
.
\end{equation}
For any \(y\in\mathbb{R}^d\), \(k\in\mathbb{Z}\) we have
\begin{align*}
y&\in A_k
&&\implies
&&y\in\bigl(\bigcup\Q\bigr)(\lambda2^k/(4d\Lambda))\setminus\bigl(\bigcup\Q\bigr)(2^{k+1})
\\
&&&\iff
&&
\lambda 2^k/(4d\Lambda)<\dist(y,\mathbb{R}^d\setminus\bigcup\Q)\leq2^{k+1}
\\
&&&\iff
&&
k+\log_2(\lambda/(4d\Lambda))<\log_2(\dist(y,\mathbb{R}^d\setminus\bigcup\Q))\leq k+1
\\
&&&\iff
&&
0\leq k+1-\log_2(\dist(y,\mathbb{R}^d\setminus\bigcup\Q))<1-\log_2(\lambda/(4d\Lambda))
.
\end{align*}
That means for any \(y\in\mathbb{R}^d\) there are at most
\[
1-\log_2(\lambda/(4d\Lambda))+1=\log_2(16d\Lambda/\lambda)
\]
many different \(k\in\mathbb{Z}\) with \(y\in A_k\).
Together with \(A_k\subset\bigcup\ti\Q\) we obtain
\begin{align}
\notag
\sum_{k\in\mathbb{Z}}
\sm{\mb E\cap A_k}
&=
\int_{\mb E}
\sum_{k\in\mathbb{Z}}
\ind{A_k}(y)
\intd\sm y
\\
\notag
&\leq
\int_{\mb E\cap\bigcup\ti\Q}
\log_2(16d\Lambda/\lambda)
\intd\sm y
\\
\notag
&=
\log_2(16d\Lambda/\lambda)
\smb{\mb E\cap\bigcup\ti\Q}
\\
\label{eq_boundedk}
&=
\bigl[
\log_2(16d\Lambda)-\log_2(\lambda)
\bigr]
\smb{\mb E\cap\bigcup\ti\Q}
.
\end{align}
Chaining the previous inequalities \cref{eq_boundarybyballs,eq_ballfromboundarycover,eq_ballssamescale,eq_boundedk} finishes the proof.
\end{proof}

Observe the extra factor \((1-\log\lambda)\) in \cref{pro_levelsets_finite_g} in comparison to \cref{lem_isoperimetricboundedaway}.
In \cite[Proposition~5.3]{weigt2020variation} we managed to remove this \(\log\) factor in the case of balls with additional effort.
For our purposes here the rate in \(\lambda\) in \cref{pro_levelsets_finite_g} is particularly irrelevant since we will apply it only for a few fixed values \(\lambda>0\) that depend only on the dimension \(d\).

As a corollary of \cref{pro_levelsets_finite_g} we can bound the first summand in \cref{eq_decomposition}.
Recall \cref{defi_qli}.

\begin{cor}
\label{cor_highdensitycubes}
Let \(L\geq2\), \(\Lambda\geq1\), let \(\Q\) be a finite set of \(\Lambda\)-stars and let \(f\in L^1(\mathbb{R}^d)\).
Then
\[
\smb{
\tbb{
\bigcup\Q_0^\lambda
\cup
\bigcup\Q_1^\lambda
}
\setminus
\mc{\{f\geq\lambda\}}
}
\lesssim_{d,L,\Lambda}
\smb{
\mb{\{f\geq\lambda\}}
\cap
\bigcup\ti{\Q^\lambda}
}
.
\]
\end{cor}
\begin{proof}
By \cref{lem_boundaryofunion} we have
\[
\tbb{
\bigcup\Q_0^\lambda
\cup
\bigcup\Q_1^\lambda
}
\subset
\tb{
\bigcup\Q_1^\lambda
}
\setminus
\mc{
\bigcup\Q_0^\lambda
}
\cup
\tb{
\bigcup\Q_0^\lambda
}
,
\]
by \cref{pro_levelsets_finite_g} we have
\[
\smb{\tb{\bigcup\Q_0^\lambda}\setminus\mc{\{f\geq\lambda\}}}
\lesssim_{d,L,\Lambda}
\smb{
\mb{\{f\geq\lambda\}}
\cap
\bigcup\ti{\Q_0^\lambda}
}
\]
and by \cref{pro_levelsets_finite_g,lem_boundaryofunion} we have
\begin{align*}
&\smb{
\tb{\bigcup\Q_1^\lambda}
\setminus
\mc{\bigcup\Q_0^\lambda\cup\{f\geq\lambda\}}
}
\lesssim_{d,\Lambda}
\smb{\mbb{\bigcup\Q_0^\lambda\cup\{f\geq\lambda\}}
\cap
\bigcup\ti{\Q_1^\lambda}
}
\\
&\qquad\leq
\smb{\mb{\bigcup\Q_0^\lambda}
\setminus
\mc{\{f\geq\lambda\}}
\cap
\bigcup\ti{\Q_1^\lambda}
}
+
\smb{
\mb{\{f\geq\lambda\}}
\cap
\bigcup\ti{\Q_1^\lambda}
}
.
\end{align*}
We use that the measure theoretic boundary is contained in the topological boundary and combine the three previous displays to finish the proof.
\end{proof}

We integrate \cref{eq_decomposition,cor_highdensitycubes} over \(\lambda\in\mathbb{R}\) and obtain
\begin{equation}
\label{eq_splitintolowandhighdensity}
\begin{aligned}
\int_{-\infty}^\infty
\smb{\tb{
\bigcup\Q^\lambda
}
\setminus
\mc{
\{f\geq\lambda\}}
}
\intd\lambda
&\leq
\int_{-\infty}^\infty
\smb{
\tb{
\bigcup\Q_2^\lambda
}
\setminus
\tc{
\bigcup
(\Q_0^\lambda\cup\Q_1^\lambda)
}
}
\intd\lambda
\\
&\qquad
+
C_{d,L,\Lambda}
\int_{-\infty}^\infty
\smb{\mb{\{f\geq\lambda\}}
\cap
\bigcup\ti{\Q^\lambda}
}
\intd\lambda
.
\end{aligned}
\end{equation}

\subsection{Reducing to almost dyadically structured stars}
\label{sec_dyadic}

The purpose of this \lcnamecref{sec_dyadic} is to represent the first term on the right hand side of \cref{eq_splitintolowandhighdensity} by almost dyadically structured stars.

\begin{defi}
\label{defi_maximalsets}
For a set \(\Q\) of measurable subsets of \(\mathbb{R}^d\) we define the \emph{maximal} sets in \(\Q\) by
\[
\ms\Q
=
\{
Q\in\Q:
\forall P\in\Q\ 
\lnot
Q\msubsetneq P
\}
.
\]
\end{defi}

\begin{lem}
\label{lem_maxcover}
Let \(\Q\) be a finite set of measurable subsets of \(\mathbb{R}^d\).
Then for any \(Q\in\Q\) there is a \(P\in\ms\Q\) with \(Q\msubset P\).
\end{lem}

\begin{proof}
Let \(Q_0\in\Q\).
Any sequence of sets \(Q_0,Q_1,\ldots\in\Q\) with \(Q_k\msubsetneq Q_{k+1}\) is finite.
Let \(Q_0,\ldots,Q_n\) be a longest such sequence.
Then there is no \(Q\in\Q\) with \(Q_n\msubsetneq Q\) which means \(\Q_n\in\ms\Q\) and
\(
Q_0\msubset Q_n
,
\)
finishing the proof.
\end{proof}

For \(n\in\mathbb{Z}\) we say that a star \(Q\) is of \emph{scale \(n\)} if \(2^{n-1}<\rad Q\leq2^n\).

\begin{pro}
\label{pro_todyadic}
Let \(L,\Lambda\geq1\), let \(\Q\) be a finite set of \(\Lambda\)-stars and let \(f\in L^1(\mathbb{R}^d)\).
Then \(\Q\) has a subset \(\S\subset\Q\) with the following properties:
\begin{enumerate}
\item
\label{it_dyadicstarsonlymax}
For any \(Q\in\S\) there is no \(P\in\Q\) with \(Q\msubsetneq P\) and \(f_P\geq f_Q\).
\item 
\label{it_starsalmostdyadic}
Let \(Q,R\in\S\) be distinct with \(\rad R\leq\rad Q\).
Then
\[
\lm{R\cap Q}<2^{-1}
\min\bigl\{
\lm{\bll Q},\lm{\bll R}
\bigr\}
,
\]
or \(R\) is both of strictly smaller scale than \(Q\)
and \(f_R>f_Q\).
\item
\label{eq_significantmass}
For \(Q\in\Q\) let
\[
\lambda_Q
=
\inf\bigl\{
\lambda:\lm{Q\cap\{f\geq\lambda\}}<(4L)^{-1}\lm Q
\bigr\}
.
\]
Then for all \(Q\in\S\) we have \(f_Q\geq\lambda_Q\) and
\[
\int_{-\infty}^\infty
\smb{
\tb{
\bigcup\Q_2^\lambda
}
\setminus
\tc{
\bigcup
(\Q_0^\lambda\cup\Q_1^\lambda)
}
}
\intd\lambda
\lesssim_{d,\Lambda}
\sum_{Q\in\S}
(f_Q-\lambda_Q)\sm{\tb Q}
.
\]
\end{enumerate}
\end{pro}

\begin{proof}
Let
\[
\tilde\Q
=
\{Q\in\Q:\forall P\in\Q\ (\lnot Q\msubsetneq P)\tx{ or }f_P< f_Q\}
.
\]
We construct \(\S\) using the following algorithm:
\begin{alg*}
Initiate \(\R=\tilde\Q\) and \(\S=\emptyset\).
We iterate the following procedure:

If \(\R\) is empty then output \(\S\) and stop.
If \(\R\) is nonempty let \(n\) be the largest scale of a star in \(\R\),
take a star \(S\in\R\) of scale \(n\) which attains
\[
\max\{f_Q:Q\in\R,\ \rad Q>2^{n-1}\}
\]
and add it to \(\S\).
Then remove all stars \(Q\) with
\(f_Q\leq f_S\)
and
\[
\lm{
Q\cap S
}
\geq
2^{-1}\min\{\lm{\bll Q},\lm{\bll S}\}
\]
from \(\R\) and repeat.

Finally, remove those sets \(Q\) from \(\S\) with \(f_Q<\lambda_Q\).
\end{alg*}

The definition of \(\tilde\Q\) and \(\S\subset\tilde\Q\) imply \cref{it_dyadicstarsonlymax}.
Next, we prove \cref{it_starsalmostdyadic}.
In each iteration in the above loop we remove at least the star \(S\) from \(\R\) that we added to \(\S\).
Since \(\tilde\Q\) is finite this means the algorithm will terminate and return a set \(\S\) of stars.
Let \(S,T\in\S\) be distinct stars with \(\rad S\leq\rad T\) and \(\lm{S\cap T}\geq2^{-1}\lm{\bll S}\).
Then \(S\) and \(T\) cannot be of the same scale, because otherwise one of them would have been removed when the other was added to \(\S\).
That means \(S\) is of strictly smaller scale than \(T\), which means \(T\) has been added to \(\S\) in an earlier step than \(S\),
and thus \(f_S>f_T\) because otherwise \(S\) would have been removed in that step.

It remains to prove \cref{eq_significantmass}.
For \(\lambda\in\mathbb{R}\) and \(S\in\S\) denote
\[
\Q_S^\lambda
=
\bigl\{
Q\in\Q_2^\lambda\cap\tilde\Q:\lm{Q\cap S}
\geq2^{-d-1}
\lm{\bll Q}
\bigr\}
.
\]
Let \(\lambda\in\mathbb{R}\) and \(Q\in\Q_2^\lambda\cap\tilde\Q\).
Then there is a star \(S\) which has been added to \(\S\) when \(Q\) was removed from \(\R\).
This means \(S\) has at least the same scale as \(Q\)
and
\[
\lm{Q\cap S}
\geq2^{-1}
\min\{\lm{\bll Q},\lm{\bll S}\}
\geq
2^{-d-1}
\lm{\bll Q}
\geq
2^{-d-1}
\Lambda^{-d}
\lm Q
\]
and \(f_S\geq f_Q\geq\lambda\).
Since we assumed \(Q\not\in\Q_0^\lambda\cup\Q_1^\lambda\) this requires \(S\not\in\Q_0^\lambda\) and thus \(\lambda_S<\lambda\leq f_S\).
That means for each \(\lambda\in\mathbb{R}\) we have
\begin{equation}
\label{eq_groupbyS}
\Q_2^\lambda
\cap
\tilde\Q
=
\bigcup_{S\in\S:\lambda_S<\lambda\leq f_S}
\Q_S^\lambda
.
\end{equation}
By \cref{lem_smallclosure}\cref{it_startcnearti}, \cref{lem_maxcover} and 
\(
\ms{\Q^\lambda}
\subset
\{Q\in\tilde\Q:f_Q\geq\lambda\}
\subset
\Q^\lambda
\)
we have
\[
\tc{
\bigcup
\Q^\lambda
}
=
\mc{
\bigcup
\Q^\lambda
}
=
\mc{
\bigcup
\ms{\Q^\lambda}
}
\subset
\tc{
\bigcup
\ms{\Q^\lambda}
}
\subset
\tc{
\bigcup
\{Q\in\tilde\Q:f_Q\geq\lambda\}
}
\subset
\tc{
\bigcup
\Q^\lambda
}
\]
and hence
\[
\tb{
\bigcup
\Q^\lambda
}
\subset
\tb{
\bigcup
\{Q\in\tilde\Q:f_Q\geq\lambda\}
}
.
\]
Abbreviating \(C=\tc{
\bigcup
(\Q_0^\lambda\cup\Q_1^\lambda)
}
\)
we can conclude
\begin{equation}
\label{eq_intildeQenough}
\tb{
\bigcup
\Q_2^\lambda
}
\setminus
C
=
\tb{
\bigcup
\Q^\lambda
}
\setminus
C
\subset
\tb{
\bigcup
(\Q^\lambda\cap\tilde\Q)
}
\setminus
C
=
\tb{
\bigcup
(\Q_2^\lambda\cap\tilde\Q)
}
\setminus
C
.
\end{equation}
Let \(S\in\S\) with \(\lambda_S<\lambda\leq f_S\).
That means \(S\not\in\Q_0^\lambda\) and thus either we have \(S\in\Q_2^\lambda\) which implies \(S\in\Q_S^\lambda\),
or we have \(S\in\Q_1^\lambda\).
By \cref{lem_boundaryofunion} in the first case or set inclusion in the second case we obtain
\begin{equation}
\label{eq_individualS}
\tb{
\bigcup
\Q_S^\lambda
}
\setminus
\tc{
\bigcup
\Q_1^\lambda
}
\subset
\Bigl[
\tb{
\bigcup
\Q_S^\lambda
}
\setminus
\tc S
\Bigr]
\cup
\tb S
.
\end{equation}
We can conclude from \cref{eq_groupbyS,eq_intildeQenough,eq_individualS,pro_levelsets_finite_g} that
\begin{align*}
\smb{
\tb{
\bigcup\Q_2^\lambda
}
\setminus
\tc{
\bigcup
(\Q_0^\lambda\cup\Q_1^\lambda)
}
}
&\leq
\sum_{S\in\S:\lambda_S<\lambda\leq f_S}
\smb{
\tb{
\bigcup
\Q_S^\lambda
}
\setminus
\tc{
\bigcup
(\Q_0^\lambda\cup\Q_1^\lambda)
}
}
\\
&\leq
\sum_{S\in\S:\lambda_S<\lambda\leq f_S}
\smb{
\tb{
\bigcup
\Q_S^\lambda
}
\setminus
\tc S
}
+
\sm{
\tb S
}
\\
&\lesssim_{d,\Lambda}
\sum_{S\in\S:\lambda_S<\lambda\leq f_S}
\sm{
\tb S
}
.
\end{align*}
Integrating both sides over \(\lambda\) implies \cref{eq_significantmass}, finishing the proof.
\end{proof}

\begin{lem}
\label{lem_alsoscaleddisjoint}
Let \(r>0\) and let \(\B\) be a set of balls \(B\) with \(r/2<\rad B\leq r\) such that for any two distinct \(B_1,B_2\in\B\) we have
\[
\lm{B_1\cap B_2}\leq2^{-1}\min\{\lm{B_1},\lm{B_2}\}
.
\]
Then for any \(c>0\) and any point \(x\in\mathbb{R}^d\),
there are at most \((8d(1+c))^d\) many balls in \(\{cB:B\in\B\}\) which contain \(x\).
\end{lem}
\begin{proof}
Let \(B(x_1,r_1),B(x_2,r_2)\in\B\) with \(r_1\leq r_2\).
Then by \cref{lem_volumeinnerstar} we have
\begin{align*}
\lm{B(x_1,r_1)}
/2
&\leq
\lm{B(x_1,r_1)\setminus B(x_2,r_2)}
\\
&\leq
\lm{B(x_1,r_1)\setminus B(x_1,r_2-|x_1-x_2|)}
\\
&\leq
\f{
d(r_1-r_2+|x_1-x_2|)
}{
r_1
}
\lm{B(x_1,r_1)}
\end{align*}
which implies
\[
d|x_1-x_2|
\geq
dr_2
-(d-1/2)r_1
\geq
r_2/2
\]
which means \(B(x_1,r_1/(4d))\) and \(B(x_2,r_2/(4d))\) are disjoint.
Then for any \(B\in\B\) with \(x\in cB\) we have
\begin{align*}
(4d)^{-1}B&\subset B\subset B(x,(1+c)r(B))\subset B(x,(1+c)r)
,\\
\lm{(4d)^{-1}B}&\geq(8d(1+c))^{-d}\lm{B(x,(1+c)r)}
.
\end{align*}
We can conclude that there are at most \((8d(1+c))^d\) many such balls.
\end{proof}

\subsection{A sparse mass estimate}
\label{sec_sparsemassabove}

Throughout this \lcnamecref{sec_sparsemassabove} let \(L\geq2,0<\alpha\leq1\), let \(Q_0\subset\mathbb{R}^d\) be bounded and let \(f\in L^1(Q_0)\).
Moreover, let \(\Q\ni Q_0\) be a finite set of sets \(Q\msubset Q_0\) such that for any two \(Q,P\in\Q\) we have
\[
\lm{Q\cap P}\in\{0,\lm Q,\lm P\}
.
\]
Assume that whenever there are \(Q,P\in\Q\) with \(P\msubsetneq Q\) for which there exists no \(R\in\Q\) with \(P\msubsetneq R\msubsetneq Q\), the set \(Q\) has an \(L\)-decomposition \(\P(Q)\) according to \cref{defi_decomposition}\cref{it_decomposition} with \(P\in\P(Q)\subset\Q\).
This means if \(Q\in\Q\) has no \(L\)-decomposition into sets in \(\Q\) then there exists no \(P\in\Q\) with \(P\msubsetneq Q\), and we write \(\P(Q)\cap\Q=\emptyset\) in this case.

The prime example is that \(Q_0\) is a dyadic cube and \(\Q\ni Q_0\) is a subset of the set of dyadic subcubes of \(Q_0\) such that for each \(Q\in\Q\) all dyadic siblings and ancestors of \(Q\) belong to \(\Q\) as well.
This example satisfies the above assumptions for \(L=2^d\).
We will apply the subsequent results later in the setting of \(\Lambda\)-stars, although the shape of the sets in \(\Q\) is mostly irrelevant for the arguments in this \lcnamecref{sec_sparsemassabove}.

For \(\R\subset\Q\) recall the definition of its maximal sets, \(\ms\R\), \cref{defi_maximalsets}.
For any \(\lambda\in\mathbb{R}\) denote 
\[
\ml\Q\lambda
=
\ms{\{
Q\in\Q:
f_Q>\lambda
\tx{ or }
\P(Q)\cap\Q=\emptyset
\}}
.
\]

Our goal is to prove the following \lcnamecref{eq_goallambda}:

\begin{pro}[{\cite[Proposition~3.1]{weigt2020variationdyadic}} for more general sets]
\label{eq_goallambda}
For any \(\mu<f_{Q_0}\) with
\[
\lm{Q_0\cap\{f>\mu\}}\leq\alpha\lm{Q_0}
\]
we have
\[
(f_{Q_0}-\mu)\lm{Q_0}
\leq\f1\alpha
\int_\mu^\infty
\lmb{
\{f>\lambda\}
\cap
\bigcup
\bigl\{
Q\in\ml\Q\lambda:
\lm{\{f\geq\lambda\}\cap Q}
\leq
L\alpha
\lm Q
\bigr\}
}
\intd\lambda
.
\]
\end{pro}

We will later need a slight modification of \cref{eq_goallambda}.

\begin{cor}[{\cite[Corollary~3.3]{weigt2020variationdyadic}} for more general sets]
\label{cla_mostmasssparseabove}
For any \(\mu<f_{Q_0}\) with
\[
\lm{Q_0\cap\{f>\mu\}}\leq(\alpha/2)\lm{Q_0}
\]
we have
\[
(f_{Q_0}-\mu)\lm{Q_0}
\leq\f2\alpha
\int_{f_{Q_0}}^\infty
\lmb{
\{f>\lambda\}
\cap
\bigcup
\bigl\{
Q\in\ml\Q\lambda:
\lm{\{f\geq\lambda\}\cap Q}
\leq
L\alpha
\lm Q
\bigr\}
}
\intd\lambda
.
\]
\end{cor}

\begin{proof}
By assumption
\begin{equation}
\label{eq_lambdalfq}
\int_\mu^{f_{Q_0}}
\lm{\{f>\lambda\}\cap Q_0}
\intd\lambda
\leq
\f\alpha2
(f_{Q_0}-\mu)\lm{Q_0}
.
\end{equation}
By \cref{eq_goallambda} we have
\begin{align*}
(f_{Q_0}-\mu)\lm{Q_0}
&\leq\f1\alpha
\int_\mu^\infty
\lmb{
\{f>\lambda\}
\cap
\bigcup
\bigl\{
Q\in\ml\Q\lambda:
\lm{\{f\geq\lambda\}\cap Q}
\leq
L\alpha
\lm Q
\bigr\}
}
\intd\lambda
\\
&\leq\f1\alpha
\int_\mu^{f_{Q_0}}
\lm{\{f>\lambda\}\cap Q_0}
\intd\lambda
\\
&\qquad+
\f1\alpha
\int_{f_{Q_0}}^\infty
\lmb{
\{f>\lambda\}
\cap
\bigcup
\bigl\{
Q\in\ml\Q\lambda:
\lm{\{f\geq\lambda\}\cap Q}
\leq
L\alpha
\lm Q
\bigr\}
}
\intd\lambda
.
\end{align*}
Now we apply \cref{eq_lambdalfq} to the first term on the right hand side and subtract \((f_{Q_0}-\mu)\lm{Q_0}/2\) from both sides.
This finishes the proof.
\end{proof}

We will prove \cref{eq_goallambda} with the help of the following \lcnamecrefs{pro_massabovewithoutabovehighdensity}:

\begin{lem}
\label{lem_usedlowdensitycubescover}
For every \(\lambda\in\mathbb{R}\) we have
\(
Q_0
\meq
\bigcup\ml\Q\lambda
.
\)
\end{lem}

\begin{proof}
By assumption on \(\Q\) for every \(Q\in\Q\) with \(\P(Q)\cap\Q\neq\emptyset\) we have \(\P(Q)\subset\Q\).
Since \(\Q\) is finite and covers \(Q_0\) this implies that \(\{Q\in\Q:\P(Q)\cap\Q=\emptyset\}\) covers \(Q_0\).
An application of \cref{lem_maxcover} finishes the proof.
\end{proof}

\begin{lem}
\label{lem_highdensity}
Let \(Q\in\Q\) with
\[
\lm{Q\cap\{f\geq f_Q\}}\geq\alpha\lm Q
.
\]
Then for any \(\mu\leq f_Q\) we have
\[
(f_Q-\mu)\lm Q
\leq
\f1\alpha
\int_\mu^{f_Q}
\lm{Q\cap\{f>\lambda\}}
\intd\lambda
.
\]
\end{lem}

\begin{proof}
\[
(f_Q-\mu)\lm Q
=
\int_\mu^{f_Q}
\lm Q
\intd\lambda
\leq
\f1\alpha
\int_\mu^{f_Q}
\lm{Q\cap\{f>\lambda\}}
\intd\lambda
.
\qedhere
\]
\end{proof}

For \(\mu\in\mathbb{R}\) take
\[
\bar\Q_\mu
\subset
\msb{\bigl\{
Q\in\Q:
f_Q\leq\mu
\tx{ or }
\lm{\{f\geq f_Q\}\cap Q}
\geq\alpha
\lm Q
\bigr\}}
\]
such that for every \(Q\) on the right hand side there of the above display there is a \(P\in\bar\Q_\mu\) with \(P\meq Q\),
and for all \(Q,P\in\bar\Q_\mu\) with \(Q\neq P\) we have \(Q\not\meq P\).

\begin{lem}
\label{pro_massabovewithoutabovehighdensity}
For any \(\mu<f_{Q_0}\) we have
\[
(f_{Q_0}-\mu)\lm{Q_0}
\leq
\f1\alpha
\int_\mu^\infty
\lmb{
\{f>\lambda\}
\setminus
\bigcup\{
Q\in\bar\Q_\mu:
f_Q
\leq
\lambda
\}
}
\intd\lambda
.
\]
\end{lem}

\begin{proof}
By Cavalieri's principle, for any measurable \(\Omega\subset\mathbb{R}^n\) and \(g\in L^1(\Omega)\) we have
\[
\int_\Omega
(g-\mu)
=
\int_\mu^\infty
\lm{\Omega\cap\{g>\lambda\}}
\intd\lambda
-
\int_{-\infty}^\mu
\lm{\Omega\cap\{g<\lambda\}}
\intd\lambda
\leq
\int_\mu^\infty
\lm{\Omega\cap\{g>\lambda\}}
\intd\lambda
.
\]
Note, that
\[
\int_Q(f-\mu)=(f_Q-\mu)\lm Q
.
\]
Thus, by the previous display for \(\Omega=Q_0\setminus\bigcup\bar\Q_\mu\) and \(g=f\) and
by applying \cref{lem_highdensity} to each \(Q\in\bar\Q_\mu\) with \(f_Q\geq\mu\) we obtain
\begin{align*}
(f_{Q_0}-\mu)\lm{Q_0}
&=
\int_{Q_0}(f-\mu)
=
\int_{Q_0\setminus\bigcup\bar\Q_\mu}
(f-\mu)
+
\sum_{Q\in\bar\Q_\mu}
\int_Q
(f-\mu)
\\
&\leq
\int_\mu^\infty
\lm{\{f>\lambda\}\setminus\bigcup\bar\Q_\mu}
\intd\lambda
+
\f1\alpha
\sum_{Q\in\bar\Q_\mu}
\int_\mu^{\max\{\mu,f_Q\}}
\lm{\{f>\lambda\}\cap Q}
\intd\lambda
\\
&=
\int_\mu^\infty
\Bigl[
\lmb{
\{f>\lambda\}
\setminus
\bigcup\bar\Q_\mu
}
+
\f1\alpha
\lmb{
\{f>\lambda\}
\cap
\bigcup\{
Q\in\bar\Q_\mu:
f_Q>\lambda
\}
}
\Bigr]
\intd\lambda
\\
&\leq
\f1\alpha
\int_\mu^\infty
\lmb{
\{f>\lambda\}
\setminus
\bigcup\{
Q\in\bar\Q_\mu:
f_Q
\leq
\lambda
\}
}
\intd\lambda
.
\qedhere
\end{align*}
\end{proof}

\begin{lem}
\label{lem_usedlowdensitycubeshighdensitysmall}
Let \(\mu\in\mathbb{R}\) with
\[
\lm{\{f>\mu\}\cap Q_0}\leq\alpha\lm{Q_0}
\]
and \(\lambda>\mu\).
Then for every \(Q\in\ml\Q\lambda\) with
\[
\lm{\{f\geq\lambda\}\cap Q}
>
L\alpha
\lm Q
\]
there is a \(P\in\bar\Q_\mu\) with \(Q\msubset P\) and \(f_P\leq\lambda\).
\end{lem}

\begin{proof}
The assumption implies \(Q\msubsetneq Q_0\).
Thus there is an \(R\in\Q\) with smallest Lebesgue measure such that \(Q\msubsetneq R\).
This means \(Q\in\P(R)\neq\emptyset\).
Since \(Q\in\ml\Q\lambda\) we must have \(R\not\in\ml\Q\lambda\).
This implies \(f_R\leq\lambda\) and hence
\[
\lm{\{f>f_R\}\cap R}
\geq
\lm{\{f\geq\lambda\}\cap Q}
>
L\alpha\lm Q
\geq
\alpha\lm R
.
\]
Thus, by \cref{lem_maxcover} there is a \(P\in\bar\Q_\mu\) with \(Q\msubsetneq R\msubset P\).
Again, we must have \(P\not\in\ml\Q\lambda\) and thus \(f_P\leq\lambda\).
\end{proof}

\begin{proof}[Proof of \cref{eq_goallambda}]
This follows from applying \cref{lem_usedlowdensitycubeshighdensitysmall,lem_usedlowdensitycubescover} to the right hand side in \cref{pro_massabovewithoutabovehighdensity}.
\end{proof}

\begin{rem}
In order to prove \cref{theo_goal} for the Hardy-Littlewood maximal operator \(\Mu\) over balls it is enough to prove a suitable variant of \cref{eq_goallambda} for the set of all balls, see also \cref{rem_balls}\cref{it_massaboveballs}.
Our proof of \cref{eq_goallambda} does not work for the set of all balls because a ball cannot be decomposed into finitely many smaller balls.
In fact it is not even entirely clear how the statement of \cref{eq_goallambda} should be formulated in the case of balls.
Note that in its current form \cref{eq_goallambda} only takes into account the parts of \(f\) that are contained within \(Q_0\), and lie above \(\mu\).
This is not strictly necessary, in principle we may allow the right hand side in \cref{eq_goallambda} to take \(f\) into account also anywhere below \(\mu\) and within \(2B_0\) in order to hopefully enable a proof of a variant for balls.
\end{rem}

\begin{lem}
\label{cor_significantmassabove}
Let \(\P\) be a countable set of sets \(Q\subset\mathbb{R}^d\) which have
a ball \(\bll Q\) with \(\lm{\bll Q}\leq\lm Q\) and \(Q\msubset \Lambda\tc{\bll Q}\)
and an \(L\)-decomposition \(\P(Q)\subset\P\) according to \cref{defi_decomposition}\cref{it_decomposition}.
Let \(Q_0\in\P\) and let \(E\subset Q_0\) be a measurable set with \(\lm{E\cap Q_0}\leq\lm{Q_0}/2\).
Then the sets \(Q\in\P\) with \(Q\msubset Q_0\) and
\[
\f1{2L}
\leq
\f{\lm{E\cap Q}}{\lm Q}
\leq
\f12
\]
cover almost all of \(E\cap Q_0\).
\end{lem}

\begin{proof}
By our assumptions on \(\P\) for almost every \(x\in E\cap Q_0\) there exists a sequence \(Q_0=Q^x_0,Q^x_1,Q^x_2,\ldots\in\P\) such that for each \(n\in\mathbb{N}\) we have \(x\in Q^x_n\) and \(Q^x_{n+1}\in\P(Q^x_n)\).
Since \(\lm{\bll{Q^x_n}}\leq\lm{Q^x_n}\) and \(Q^x_n\msubset \Lambda\tc{\bll{Q^x_n}}\) we can infer from the Lebesgue density theorem that for almost every \(x\in E\cap Q_0\) we have \(\lm{E\cap Q^x_n}/\lm{Q^x_n}\to 1\) for \(n\rightarrow\infty\).
Let \(n_x\) be the smallest integer \(n\) for which \(2L\lm{E\cap Q^x_n}\geq\lm{Q^x_n}\).
If \(n_x=0\) then \(Q^x_{n_x}=Q_0\) which means \(2\lm{E\cap Q^x_{n_x}}\leq\lm{Q^x_{n_x}}\) by assumption.
In case \(n_x\geq1\) we argue
\[
\f{
\lm{E\cap Q^x_{n_x}}
}{
\lm{Q^x_{n_x}}
}
\leq
\f{
\lm{E\cap Q^x_{n_x-1}}
}{
\lm{Q^x_{n_x}}
}
\leq
\f{
\lm{Q^x_{n_x-1}}/(2L)
}{
\lm{Q^x_{n_x}}
}
\leq
\f12
,
\]
finishing the proof.
\end{proof}

\subsection{Organizing mass}
\label{sec_organizingmass}

Recall that for \(t\geq0\) and a set \(Q\subset\mathbb{R}^d\) we define
\[
Q(t)
=
\{x\in Q:\dist(x,\mathbb{R}^d\setminus Q)>t\}
.
\]

\begin{lem}
\label{lem_disjointmass}
For every \(\Lambda\geq 1\) and every \(0<\varepsilon<1\) there exist constants \(C,C_1,C_2>0\) which allow for the following:
Let \(\Q\) and \(\S\) be finite collections of sets \(Q\subset\mathbb{R}^d\) which have a ball \(\bll Q\) with \(\bll Q\subset Q\subset \Lambda\tc{\bll Q}\) and denote by \(\rad Q\) the radius of \(\bll Q\).
Assume that for each \(Q_0\in\S\) there is a collection \(\Q_{Q_0}\subset\Q\) with \(\bigcup\Q_{Q_0}\msubset Q_0\)
and that for each \(Q\in\Q\) we have \(\lnot Q_0\msubsetneq Q\).
Then \(\Q\) has a subset \(\P\) with the following properties:
\begin{enumerate}
\item
\label{it_boundedintersect}
For any \(x\in\mathbb{R}^d\) there are at most \(C\) many sets \(Q\in\P\) with \(x\in Q(\varepsilon\rad Q)\).
\item
\label{it_insmallcubes}
For each \(Q_0\in\S\) and \(Q\in\Q_{Q_0}\)
there exists a set \(P\in \P\) with \(Q\subset C_1\tc{\bll P}\)
and \(P\subset C_2\tc{\bll{Q_0}}\).
\end{enumerate}
\end{lem}

\begin{proof}
Denote by \(\tilde\Q\) the set of all \(Q\in\Q\) for which there is no \(P\in\Q\) with \(Q\subset P(\varepsilon\rad P/2)\).
For each \(n\in\mathbb{Z}\) denote by \(\tilde\Q_n\) the set of all \(Q\in\tilde\Q\) with \(2^{n-1}<\rad Q\leq2^n\).
Take a maximal collection \(\P_n\subset \tilde\Q_n\) such that for any two distinct \(Q,P\in\P_n\) their subsets
\(Q(\varepsilon\rad Q)\) and \(P(\varepsilon\rad P)\) are disjoint.
Set \(\P=\bigcup_{n\in\mathbb{Z}}\P_n\). 

First we prove \cref{it_boundedintersect}.
Let \(x\in\mathbb{R}^d\) and let \(n\in\mathbb{Z}\) be the largest integer for which there is a \(Q\in\P_n\) with \(x\in Q(\varepsilon\rad Q)\).
Let \(k\in\mathbb{Z}\) with \(k\leq n-(3+\log_2(\Lambda/\varepsilon))\) and \(P\in\P_k\).
Then
\[
\diam(P)
\leq
2\Lambda\rad P
\leq
2\Lambda2^k
\leq
\varepsilon2^{n-2}
<
\varepsilon\rad Q/2
.
\]
That means we cannot have \(x\in P\) because otherwise \(P\subset Q(\varepsilon\rad Q/2)\) which
contradicts \(P\in\tilde\Q\).
We can conclude that there are at most \(C=\log_2(\Lambda/\varepsilon)+4\) many integers \(k\in\mathbb{Z}\) for which there is a \(Q\in\P_k\) with \(x\in Q(\varepsilon\rad Q)\).
Since by definition of \(\P_k\) for each \(k\in\mathbb{Z}\) there is at most one such \(Q\in\P_k\)
we can conclude \cref{it_boundedintersect}.

Now we prove \cref{it_insmallcubes}.
Let \(Q_0\in\S\) and \(Q\in\Q_{Q_0}\).
If \(Q\in\P\) take \(P=Q\).
Then \(Q\subset\Lambda\tc{\bll Q}\) and \(\bll Q\msubset Q_0\subset\Lambda\tc{\bll{Q_0}}\) which imply \(Q\subset\Lambda^2\tc{\bll{Q_0}}\).
If \(Q\in\tilde\Q\setminus\P\) then there is a \(P\in\P\)
with
\(\rad Q\leq2\rad P\leq4\rad Q\leq2\diam{Q_0}\leq4\Lambda\rad{Q_0}\)
such that \(Q(\varepsilon\rad Q/2)\) intersects \(P(\varepsilon\rad P/2)\).
That means \(Q\cap P\) contains a small ball, almost all of which also belongs to \(Q_0\).
We can conclude
\(Q\subset5\Lambda\tc{\bll P}\),
\(P\subset5\Lambda\tc{\bll Q}\),
and
\(P\subset5\Lambda^2\tc{\bll{Q_0}}\).
If \(Q\not\in\tilde\Q\) then there is an \(R\in\Q\) with \(R(\varepsilon\rad R/2)\supset Q\),
and since \(\Q\) is finite such \(R\) exists for which there is no \(\tilde R\in Q\) with \(R\subset\tilde R(\varepsilon\rad{\tilde R}/2)\).
That means \(R\in \tilde\Q\) and so as in the previous case there is a \(P\in \P\) with
\(Q\subset R\subset5\Lambda\tc{\bll P}\)
and
\(P\subset5\Lambda\tc{\bll R}\).
Since \(Q\msubset Q_0\) also \(Q_0\) intersects \(R(\varepsilon\rad R/2)\).
Thus, if \(4\Lambda\rad{Q_0}<\varepsilon\rad R\)
then \(\diam(Q_0)<\varepsilon\rad R/2\) and hence \(Q_0\subset R\),
and moreover
\[
\lm{R\setminus Q_0}
\geq
\lm{\bll R}
-
\lm{\Lambda\bll{Q_0}}
\geq
\lm{\bll R}
(1-\varepsilon/4)
>0
\]
which would be a contradiction to our assumption on \(\S\) and \(\Q\).
That means we must have \(\rad R\leq4\Lambda\rad{Q_0}/\varepsilon\) and thus
\[
P\subset5\Lambda\tc{\bll R}\subset(24\Lambda/\varepsilon+1)\Lambda\tc{\bll{Q_0}}
.
\]
\end{proof}

\subsection{Combining the results}
\label{sec_combination}

\begin{lem}
\label{lem_contractalittle}
For any \(t>0\), any \(\Lambda\)-star \(Q\) with \(\lm{Q(t)}>0\) and any measurable set \(E\subset\mathbb{R}^d\) we have
\[
\Bigl|
\f{
\lm{Q(t)\cap E}
}{
\lm{Q(t)}
}
-
\f{
\lm{Q\cap E}
}{
\lm Q
}
\Bigr|
\leq
\f{2dt}{\rad Q}
.
\]
\end{lem}

\begin{proof}
Since
\[
\Bigl|
\f{
\lm{Q(t)\cap E}
}{
\lm{Q(t)}
}
-
\f{
\lm{Q\cap E}
}{
\lm Q
}
\Bigr|
\leq
\Bigl|
\f{
\lm{Q(t)\cap E}
}{
\lm{Q(t)}
}
-
\f{
\lm{Q(t)\cap E}
}{
\lm Q
}
\Bigr|
+
\Bigl|
\f{
\lm{Q(t)\cap E}
}{
\lm Q
}
-
\f{
\lm{Q\cap E}
}{
\lm Q
}
\Bigr|
,
\]
\[
\Bigl|
\f{
\lm{Q(t)\cap E}
}{
\lm{Q(t)}
}
-
\f{
\lm{Q(t)\cap E}
}{
\lm Q
}
\Bigr|
=
\f{
\lm{Q(t)\cap E}
}{
\lm{Q(t)}
\lm Q
}
|
\lm Q
-
\lm{Q(t)}
|
\leq
\f{
\lm{Q\setminus Q(t)}
}{
\lm Q
}
\]
and
\[
\Bigl|
\f{
\lm{Q(t)\cap E}
}{
\lm Q
}
-
\f{
\lm{Q\cap E}
}{
\lm Q
}
\Bigr|
\leq
\f{
\lm{Q\setminus Q(t)}
}{
\lm Q
}
,
\]
we can conclude the proof using \cref{lem_volumeinnerstar}.
\end{proof}

\begin{proof}[Proof of \cref{theo_main}]
In \cref{sec_covering} we proved \cref{eq_splitintolowandhighdensity}.
It remains to bound the first term on the right hand side of \cref{eq_splitintolowandhighdensity}, and we will proceed to do so using the tools developed in \cref{sec_dyadic,sec_sparsemassabove,sec_organizingmass}.

Let \(\S\subset\Q\) be the set from \cref{pro_todyadic}.
For every \(Q_0\in\S\) let \(\Q_{Q_0,0}=\{Q_0\}\) and for \(n\in\mathbb{N}\) inductively define
\[
\Q_{Q_0,n+1}
=
\bigcup\{\P(Q):
Q\in\Q_{Q_0,n}\ \exists P\in\Q\ P\msubsetneq Q
\}
.
\]
Since \(\Q\) is finite there is an \(N\in\mathbb{N}\) for which \(\Q_{Q_0,N+1}=\emptyset\) and we let
\[
\Q_{Q_0}=\Q_{Q_0,0}\cup\ldots\cup\Q_{Q_0,N}
.
\]
Then \(Q_0\) and \(\Q_{Q_0}\) satisfy the assumptions from the first paragraph in \cref{sec_sparsemassabove}.
Recall that \(\lambda_{Q_0}\) from \cref{pro_todyadic} satisfies
\[
\lm{Q_0\cap\{f>\lambda_{Q_0}\}}\leq(4L)^{-1}\lm{Q_0}
\]
and let
\[
\T_{Q_0}^\lambda
=
\Bigl\{
Q\in\ml{\Q_{Q_0}}\lambda:
\lm{\{f\geq\lambda\}\cap Q}
\leq
\f{
\lm Q
}2
\Bigr\}
.
\]
By \cref{cla_mostmasssparseabove} with \(\alpha=(2L)^{-1}\) we have
\begin{equation}
\label{eq_massabove}
(f_{Q_0}-\lambda_{Q_0})\lm{Q_0}
\leq4L
\int_{f_{Q_0}}^\infty
\lmb{
\{f>\lambda\}
\cap
\bigcup
\T_{Q_0}^\lambda
}
\intd\lambda
.
\end{equation}
Since \(\P\) satisfies the assumptions in \cref{cor_significantmassabove} for every \(\lambda\in\mathbb{R}\) the collection
\[
\bigcup_{
Q\in\T_{Q_0}^\lambda
}
\Bigl\{
P\in\P:
P\msubset Q
,\ 
\f1{2L}
\leq
\f{\lm{P\cap\{f\geq\lambda\}}}{\lm P}
\leq
\f12
\Bigr\}
\]
has a finite subset \(\R_{Q_0}^\lambda\) with
\begin{equation}
\label{eq_covermostoffbyfinite}
\lmb{
\bigcup
\R_{Q_0}^\lambda
}
\geq
\f9{10}
\lmb{
\{f\geq\lambda\}\cap\bigcup\T_{Q_0}^\lambda
}
.
\end{equation}

Now we show that for every \(\lambda\in\mathbb{R}\) the premise of \cref{lem_disjointmass} holds for the sets
\[
\tilde\S\seq\{Q_0\in\S:f_{Q_0}<\lambda\}
,\qquad
\tilde\Q_{Q_0}\seq\R_{Q_0}^\lambda
,\qquad
\tilde\Q\seq\bigcup_{Q_0\in\tilde\S}\tilde\Q_{Q_0}
.
\]
So let \(\lambda\in\mathbb{R}\) and \(Q,Q_0\in\S\subset\Q\) with \(f_Q,f_{Q_0}<\lambda\) and \(P\in\R_{Q_0}^\lambda\) with \(Q\msubset P\).
We need to show that \(P\msubset Q\).
Unpacking our definitions there is an \(n\in\{0,1,\ldots\}\) and a star \(R\in\ml{\Q_{Q_0}}\lambda\cap\Q_{Q_0,n}\) with \(P\msubset R\).
This means \(Q\msubset R\) and that \(f_R>\lambda\) or \(\P(R)\cap\Q_{Q_0}=\emptyset\).
If \(Q\meq R\) then \(P\msubset Q\) and we are done.
Thus it suffices to consider the case \(Q\msubsetneq R\).
Then by definition of \(\Q_{Q_0,n+1}\) we have \(\P(R)\subset\Q_{Q_0,n+1}\) which means \(\P(R)\cap\Q_{Q_0}\neq\emptyset\),
and thus \(f_R>\lambda\) must hold.
Since \(R\in\Q_{Q_0,n}\) there exists a sequence \(Q_0,\ldots,Q_n\in\P\) with \(Q_n=R\) such that for each \(0\leq k\leq n-1\) we have \(Q_k\in\Q_{Q_0,k}\) and \(Q_{k+1}\in\P(Q_k)\).
That means for every \(0\leq k\leq n\) we have \(Q\msubsetneq Q_k\).
Since \(\Q\) is \(L\)-complete it follows inductively that \(Q_0,\ldots,Q_n\in\Q\) and thus \(R\in\Q\).
But then \(Q\msubsetneq R\) and \(f_R>\lambda\geq f_Q\) contradict \cref{pro_todyadic}\cref{it_dyadicstarsonlymax}.

That means for every \(\lambda\in\mathbb{R}\) the premise holds and we can apply \cref{lem_disjointmass} with
\[
\tilde\S=\{Q_0\in\S:f_{Q_0}<\lambda\}
,\quad
\tilde\Q_{Q_0}=\R_{Q_0}^\lambda
,\quad
\tilde\Q=\bigcup_{Q_0\in\tilde\S}\tilde\Q_{Q_0}
,\quad
\varepsilon=(8dL)^{-1}
.
\]
We denote the resulting set of \(\Lambda\)-stars by \(\F^\lambda\).
By \cref{lem_disjointmass}\cref{it_insmallcubes} for every \(Q_0\in\S\) and \(\lambda>f_{Q_0}\) we have
\begin{equation}
\label{eq_toglobalmass}
\lmb{
\bigcup\R_{Q_0}^\lambda
}
\leq
\lmb{
\bigcup\{C_1\bll Q:Q\in \F^\lambda,\ Q\subset C_2\bll{Q_0}\}
}
\leq C_1^d
\sum_{Q\in\F^\lambda:Q\subset C_2\bll{Q_0}}
\lm{\bll Q}
.
\end{equation}
By \cref{cla_largeboundaryinball} for every \(Q_0\in\S\) we have
\begin{equation}
\label{eq_starperimetervsvolume}
\sm{\tb{Q_0}}
\lesssim_d\Lambda^{d+1}
\sm{\tb{\Lambda\bll{Q_0}}}
=\Lambda^{2d}
\sm{\tb{\bll{Q_0}}}
\lesssim_{d,\Lambda}
\lm{Q_0}
/
\rad{Q_0}
.
\end{equation}
By \cref{pro_todyadic}\cref{eq_significantmass} and \cref{eq_massabove,eq_covermostoffbyfinite,eq_toglobalmass,eq_starperimetervsvolume} we obtain
\begin{align}
\notag
\int_{-\infty}^\infty
\smb{
\tb{
\bigcup\Q_2^\lambda
}
\setminus
\tc{
\bigcup
(\Q_0^\lambda\cup\Q_1^\lambda)
}
}
\intd\lambda
&\lesssim_{d,\Lambda}
\sum_{Q_0\in\S}
(f_{Q_0}-\lambda_{Q_0})
\sm{\tb{Q_0}}
\\
\notag
&\lesssim_{d,L,\Lambda}
\sum_{Q_0\in\S}
\rad{Q_0}^{-1}
\int_{f_{Q_0}}^\infty
\sum_{Q\in\F^\lambda:Q\subset C_2\bll{Q_0}}
\lm{\bll Q}
\intd\lambda
\\
\label{eq_tomassabove}
&\leq
\int_{-\infty}^\infty
\sum_{Q\in \F^\lambda}
\lm{\bll Q}
\sum_{Q_0\in\S,Q\subset C_2\bll{Q_0}}
\rad{Q_0}^{-1}
\intd\lambda
.
\end{align}
Let \(k\in\mathbb{Z}\) and let \(Q_0,Q_1\in\S\) be distinct with
\[
2^{k-1}<\rad{Q_0}\leq\rad{Q_1}\leq2^k
.
\]
Then by \cref{pro_todyadic}\cref{it_starsalmostdyadic} we have
\[
\lm{\bll{Q_0}\cap\bll{Q_1}}
\leq2^{-1}
\min\{\lm{\bll{Q_0}},\lm{\bll{Q_1}}\}
.
\]
Let \(Q\in \F^\lambda\) and let \(n\in\mathbb{Z}\) with \(2^{n-1}<\rad Q\leq 2^n\).
Then by \cref{lem_alsoscaleddisjoint} for any \(k\in\mathbb{Z}\) the number of \(\Lambda\)-stars \(Q_0\in\S\) with \(2^{k-1}<\rad{Q_0}\leq2^k\) and \(Q\subset C_2\bll{Q_0}\) is bounded by a constant that depends only on \(d,L,\Lambda\).
Moreover, if \(k\leq n-1-\log_2C_2\) then
\[
\diam(C_2\bll{Q_0})
\leq
2C_2\rad{Q_0}
\leq
2^{k+1}C_2
\leq
2^n<2\rad Q
\leq
\diam(Q)
\]
which means \(Q\not\subset C_2\bll{Q_0}\).
We can conclude
\begin{equation}
\label{eq_massbelowsums}
\sum_{
\substack{
Q_0\in\S
,\\
Q\subset C_2\bll{Q_0}
}}
\mkern-20mu
\rad{Q_0}^{-1}
=
\mkern-10mu
\sum_{k>n-1-\log_2C_2}
\mkern-10mu
\sum_{
\substack{
Q_0\in\S,\\
Q\subset C_2\bll{Q_0},\\
2^{k-1}<\rad{Q_0}\leq2^k
}}
\mkern-20mu
\rad{Q_0}^{-1}
\lesssim_{d,\Lambda,L}
\mkern-10mu
\sum_{k>n-1-\log_2C_2}
\mkern-20mu
2^{-k}
<
2^{2-n}C_2
\lesssim_{d,L,\Lambda}
\f1
{\rad Q}
.
\end{equation}
By \cref{lem_starinnerstar} for any \(Q\in\R_{Q_0}^\lambda\) the set \(Q(\varepsilon\rad Q)\) is a \(\Lambda/(1-\varepsilon)\)-star.
Moreover by the definition of \(\R_{Q_0}^\lambda\) we have
\[
(2L)^{-1}\leq\lm{Q\cap\{f\geq\lambda\}}/\lm Q\leq1/2
.
\]
Recalling \(\varepsilon=(8dL)^{-1}\) by \cref{lem_contractalittle} we obtain
\[
\f1{4L}
\leq
\f{
\lm{Q(\varepsilon\rad Q)\cap\{f\geq\lambda\}}
}{
\lm{Q(\varepsilon\rad Q)}
}
\leq
\f34
.
\]
That means by \cref{lem_volumeinnerstar,lem_isoperimetricboundedaway} that
\begin{equation}
\lm{\bll Q}/\rad Q
\lesssim_d
\lm Q^{(d-1)/d}
\lesssim_d
\lm{Q(\varepsilon\rad Q)}^{(d-1)/d}
\label{eq_volumetoboundary}
\lesssim_{d,\Lambda,L}
\sm{Q(\varepsilon\rad Q)\cap\mb{\{f\geq\lambda\}}}
.
\end{equation}
Recall that by \cref{lem_disjointmass}\cref{it_boundedintersect} for every \(\lambda\in\mathbb{R}\) and \(x\in\mathbb{R}^d\) there are at most \(C\) different stars \(Q\in\F^\lambda\) with \(x\in Q(\varepsilon\rad Q)\).
Moreover, for every \(Q\in\F^\lambda\) there is an \(S\in\S\) with \(Q\msubset S\) and thus by \cref{lem_smallclosure}\cref{it_startopeneqmesopen} we have
\[
Q(\varepsilon\rad Q)\subset\mi Q\subset\mi S=\ti S
.
\]
Denoting \(\bigcup\ti\S\seq\bigcup\{\ti S:S\in\S\}\) we obtain
\begin{align}
\notag
\sum_{Q\in\F^\lambda}
\sm{Q(\varepsilon\rad Q)\cap\mb{\{f\geq\lambda\}}}
&=
\int_{\mb{\{f\geq\lambda\}}}
\sum_{Q\in\F^\lambda}
\ind{Q(\varepsilon\rad Q)}(x)
\intd\sm x
\\
\notag
&\leq C
\int_{\bigcup\ti\S\cap\mb{\{f\geq\lambda\}}}
\intd\sm x
\\
\label{eq_massabovedisjoint}
&=C
\smb{\bigcup\ti\S\cap\mb{\{f\geq\lambda\}}}
.
\end{align}
Combining \cref{eq_massbelowsums,eq_volumetoboundary,eq_massabovedisjoint}
we obtain
\begin{equation}
\label{eq_eachlevelset}
\sum_{Q\in \F^\lambda}
\lm{\bll Q}
\sum_{Q_0\in\S,Q\subset C_2\bll{Q_0}}
\rad{Q_0}^{-1}
\lesssim_{d,L,\Lambda}
\smb{\bigcup\ti\S\cap\mb{\{f\geq\lambda\}}}
.
\end{equation}
We integrate \cref{eq_eachlevelset} over \(\lambda\in\mathbb{R}\) and apply \cref{eq_tomassabove} to obtain
\[
\int_{-\infty}^\infty
\smb{
\tb{
\bigcup\Q_2^\lambda
}
\setminus
\tc{
\bigcup
(\Q_0^\lambda\cup\Q_1^\lambda)
}
}
\intd\lambda
\lesssim_{d,L,\Lambda}
\int_{-\infty}^\infty
\smb{\bigcup\ti\S\cap\mb{\{f\geq\lambda\}}}
\intd\lambda
.
\]
Applying the previous inequality to \cref{eq_splitintolowandhighdensity} finishes the proof.
\end{proof}

\section{Local integrability and approximation}
\label{sec_approximation}

\subsection{Perturbations with zero Lebesgue measure}
\label{subsec_openclosed}

In this \lcnamecref{subsec_openclosed} we prove \cref{rem_tiortc}.

\begin{lem}
\label{lem_lindelof}
Every set \(\B\) of open balls has a countable subset \(\C\subset\B\) with
\(
\bigcup\C
=
\bigcup\B
.
\)
\end{lem}

\begin{proof}
Because we may write \(\B\) as the countable union
\[
\B=
\bigcup_{n\in\mathbb{Z}}
\{B\in\B:2^{n-1}<r(B)\leq 2^n\}
,
\]
after rescaling it is enough to prove the \lcnamecref{lem_lindelof} in the case that \(\B\) is a set of balls \(B\) with \(2^{-1}<r(B)\leq 1\).
Then for \(N\in2\mathbb{N}\) and \(0\leq n<N/2\) let
\(\B_n^N\) be a maximal set of balls \(B\in\B\) with
\[
2^{-1}+n/N< r(B)\leq 2^{-1}+(n+1)/N
\]
such that the centers of any two distinct balls in \(\B_n^N\) have distance at least \(1/N\).
Let \(x\in\bigcup\B\).
Then there is a ball \(B(y,r)\in\B\) with \(x\in B(y,r)\) and an \(N\in2\mathbb{N}\) with \(|y-x|\leq r-2/N\). 
Moreover, there is a \(0\leq n<N/2\) with
\[
2^{-1}+n/N< r\leq 2^{-1}+(n+1)/N
\]
which means there exists a \(B(\tilde y,\tilde r)\in\B_n^N\) such that \(|\tilde y-y|\leq1/N\).
We can conclude
\[
|x-\tilde y|
\leq
|x-y|
+
|y-\tilde y|
\leq
r-2/N+1/N
=
r-1/N
<
\tilde r
\]
which means
\(x\in B(\tilde y,\tilde r)\).
Since \(\bigcup_{N\in2\mathbb{N}}\bigcup_{0\leq n<N/2}\B_n^N\) is a countable set of balls this finishes the proof.
\end{proof}

Recall that for a set \(\Q\) of \(\Lambda\)-stars we denote
\(
\tc\Q
=
\{\tc Q:Q\in\Q\}
\)
and
\(
\ti\Q
=
\{\ti Q:Q\in\Q\}
.
\)

\begin{lem}
\label{it_starsettcvsti}
For any set \(\Q\) of \(\Lambda\)-stars the sets \(\bigcup\tc\Q\), \(\bigcup\ti\Q\) and \(\bigcup\Q\) are measurable and
\[
\bigcup\tc\Q
\meq
\bigcup\Q
\meq
\bigcup\ti\Q
.
\]
\end{lem}

\begin{proof}
The set \(\bigcup\ti\Q\) is Lebesgue measurable since it is open.
For any \(n\in\mathbb{Z}\) let
\[
\Q_n=\{Q\in\Q:2^{n-1}<\rad Q\leq 2^n\}
.
\]
Then by \cref{lem_smallclosure}\cref{it_startcnearti} we have
\[
\bigcup\tc\Q
\setminus
\bigcup\ti\Q
\subset
\bigcup_{n\in\mathbb{Z}}
\bigcup\tc{\Q_n}
\setminus
\bigcup\ti{\Q_n}
\subset
\bigcup_{n\in\mathbb{Z}}
\tb{\bigcup\ti{\Q_n}}
.
\]
Let \(n\in\mathbb{Z}\) and let \(B\) be a ball with \(r(B)>2^{n-1}\).
By \cref{lem_smallclosure}\cref{it_startcandtiarestars} the collection \(\ti{\Q_n}\) consists of \(\Lambda\)-stars \(\ti Q\) with \(\rad{\ti Q}>2^{n-1}\)
and thus we can apply \cref{cla_largeboundaryinball} and obtain \(\lm{B\cap\tb{\bigcup\ti{\Q_n}}}=0\).
Covering \(\mathbb{R}^d\) by a countable set of such balls \(B\) and summing over \(n\in\mathbb{Z}\) we thus obtain
\[
\lmb{
\bigcup\tc\Q
\setminus
\bigcup\ti\Q
}
=
0
\]
and hence
\[
\bigcup\tc\Q
\msubset
\bigcup\ti\Q
\subset
\bigcup\Q
\subset
\bigcup\tc\Q
.
\]
That means that the three previous sets have the same Lebesgue measure and in particular are all Lebesgue measurable.
\end{proof}

\begin{lem}
\label{lem_perturbationmeasurable}
Let \(\Q\) and \(\tilde\Q\) be sets as in \cref{rem_tiortc}, i.e.\
let \(\Q\) be a set of \(\Lambda\)-stars and let \(\tilde\Q\) be a set of subsets of \(\mathbb{R}^d\) such that for each \(Q\in\Q\) there exists a \(\tilde Q\in\tilde\Q\), and, conversely, for each \(\tilde Q\in\tilde\Q\) there exists a \(Q\in\Q\), such that \(\tilde Q\subset\tc Q\) and \(\ti Q\msubset\tilde Q\).
Then \(\bigcup\tilde\Q\) is Lebesgue measurable with \(\bigcup\tilde\Q\meq\bigcup\Q\).
\end{lem}

\begin{proof}
For any \(x\in\bigcup\ti\Q\) there is a \(Q_x\in\Q\) and a ball \(B_x\) with \(x\in B_x\subset\ti{Q_x}\msubset\tilde{Q_x}\).
We have \(\bigcup\ti\Q=\bigcup\bigl\{B_x:x\in\bigcup\ti\Q\bigr\}\)
and thus by \cref{lem_lindelof} there is a sequence \(x_1,x_2,\ldots\in\bigcup\ti\Q\) such that \(\bigcup\ti\Q=B_{x_1}\cup B_{x_2}\cup\ldots\)
which means
\[
\lmb{
\bigcup\ti\Q
\setminus
\bigcup\tilde\Q
}
\leq
\sum_{n=1}^\infty
\lm{
B_{x_n}
\setminus
\tilde{Q_{x_n}}
}
=
0
.
\]
That means \(\bigcup\ti\Q\msubset\bigcup\tilde\Q\subset\bigcup\tc\Q\) and we can finish the proof using \cref{it_starsettcvsti}.
\end{proof}

\begin{proof}[Proof of \cref{rem_tiortc}]
As a special case of \cref{lem_perturbationmeasurable} any star \(Q\in\Q\) differs from its corresponding set \(\tilde Q\in\tilde\Q\) by a set of Lebesgue measure zero so that \(f_{\tilde Q}=f_Q\).
By another application of \cref{lem_perturbationmeasurable} we obtain
\[
\bigcup\{Q\in\Q:f_Q>q\}
\meq
\bigcup\{\tilde Q\in\tilde\Q:f_{\tilde Q}>q\}
.
\]
Therefore, the set
\begin{align*}
\{\M_{\tilde\Q}f>\M_\Q f\}
=
\bigcup_{q\in\mathbb{Q}}
\{x\in\Omega:\M_{\tilde\Q}f(x)>q>\M_\Q f(x)\}
\end{align*}
has Lebesgue measure zero, which means that for almost every \(x\in\Omega\) we have \(\M_{\tilde\Q}f(x)\leq\M_\Q f(x)\).
The reverse inequality follows the same way.

It remains to consider \cref{theo_main} where \(\Q\) and \(\tilde\Q\) are assumed finite.
By the above arguments we have
\[
\mb{\bigcup{\tilde\Q}^\lambda}
=
\mb{\bigcup\Q^\lambda}
\subset
\
\tb{\bigcup\Q^\lambda}
\]
and for every pair \(Q\in\Q\), \(\tilde Q\in\tilde\Q\) we have
\(\ti Q\subset\mi Q=\mi{\tilde Q}\).
That means \cref{theo_main} for \(\Q\) implies \cref{theo_main} for \(\tilde\Q\) with the measure theoretic boundary and interior instead of the topological boundary and interior.

It remains to consider the case that for every pair \(Q\in\Q\), \(\tilde Q\in\tilde\Q\) we have \(\ti Q\subset\tilde Q\) which means \(\ti Q\subset\ti{\tilde Q}\).
By \cref{lem_smallclosure}\cref{it_opencompletestars} also \(\ti\Q\) is an \(L\)-complete set of \(\Lambda\)-stars and thus \cref{theo_main} holds also for \(\ti\Q\) instead of \(\Q\).
By \cref{lem_smallclosure}\cref{it_startcnearti} we have
\[
\tc{\tilde Q}
\subset
\tc Q
\subset
\tc{\ti Q}
\subset
\tc{\tilde Q}
\]
and thus
\[
\tc{\bigcup\tilde\Q^\lambda}
=
\bigcup\tc{\tilde\Q^\lambda}
=
\bigcup\tc{\ti{\Q^\lambda}}
=
\tc{\bigcup\ti{\Q^\lambda}}
.
\]
We can conclude
\[
\tb{\bigcup\tilde\Q^\lambda}
\subset
\tc{\bigcup\tilde\Q^\lambda}
\setminus
\bigcup\ti{\tilde\Q^\lambda}
\subset
\tc{\bigcup\ti{\Q^\lambda}}
\setminus
\bigcup\ti{\Q^\lambda}
=
\tb{\bigcup\ti{\Q^\lambda}}
\]
and thus \cref{theo_main} for \(\tilde\Q\) follows from \cref{theo_main} for \(\ti\Q\).
\end{proof}

\subsection{Approximating uncountable sets of stars}
\label{subsec_uncountable}

In this \lcnamecref{subsec_uncountable} we prove the local integrability of the local maximal function,
and use it to deduce \cref{theo_maininfinite} from \cref{theo_main}.
Let \(\Lambda\geq1\), let \(\Omega\subset\mathbb{R}^d\) be an open set and let \(f\in L^1_\loc(\Omega)\).
Recall the local Hardy-Littlewood maximal operator
\begin{align*}
\Mu_\Omega
&=
\M_\Q
&\tx{for}\quad
\Q
&=
\{B:B\tx{ ball},\ \tc B\subset\Omega\}
,\intertext{
and denote
}
\Mj_{\Lambda,\Omega}
&=
\M_\Q
&\tx{for}\quad
\Q
&=
\{\tc Q:Q\tx{ \(\Lambda\)-John domain},\ Q\subset\Omega\}
.
\end{align*}

\begin{pro}
\label{lem_locallylikelocal}
Let \(\Lambda\geq1\), let \(\Omega\subset\mathbb{R}^d\) be an open set and let \(f\in L^1_\loc(\Omega)\) with \(f\geq0\) and \(\var_\Omega f<\infty\).
Then for every ball \(B\) and every \(\varepsilon>0\) with \((1+\varepsilon)\tc B\subset\Omega\) there is a \(K<\infty\) such that
for every \(x\in B\) we have
\[
\Mj_{\Lambda,\Omega} f(x)\lesssim_{d,\Lambda,\varepsilon}\max\{K,\Mu_{(1+\varepsilon)B} f(x)\}
.
\]
\end{pro}

\begin{proof}
Recall \cref{lem_Johnvstar}\cref{it_Johnhasball}, i.e.\ that every \(\Lambda\)-John domain \(Q\) has a ball \(\bll Q\) with \(\bll Q\subset Q\subset\Lambda\tc{\bll Q}\).
Define
\[
\Q
=
\bigl\{
Q\ \Lambda\tx{-John domain}:
Q\subset\Omega,\ \tc Q\cap B\neq\emptyset,\ \Lambda\bll Q\setminus(1+\varepsilon/2)B\neq\emptyset
\bigr\}
\]
and
\(
K
=
\sup_{Q\in\Q}
f_Q
.
\)
Note, that by \cref{lem_smallclosure}\cref{it_startcnearti} the sets \(Q\) and \(\tc Q\) only differ by a set of Lebesgue measure zero and in particular we have \(f_Q=f_{\tc Q}\).
Let \(x\in B\) and let \(Q\) be a \(\Lambda\)-John domain with \(x\in\tc Q\) and \(\Lambda\bll Q\subset(1+\varepsilon/2)B\).
Then for the ball \(C=\f{1+2\varepsilon/3}{1+\varepsilon/2}\Lambda\bll Q\) we have
\[
x\in\tc Q\subset\Lambda\tc{\bll Q}\subset C
\]
and \(\tc C\subset(1+\varepsilon)B\) and thus
\[
\f1{\lm Q}\int_Q f
\leq
\f{(1+2\varepsilon/3)^d}{(1+\varepsilon/2)^d}
\f{\Lambda^d}{\lm C}\int_C f
\leq
\f{(1+2\varepsilon/3)^d}{(1+\varepsilon/2)^d}
\Lambda^d\Mu_{(1+\varepsilon)B}f(x)
\]
which implies
\[
\Mj_{\Lambda,\Omega} f(x)
\leq
\max\{(1+2\varepsilon/3)^d\Lambda^d\Mu_{(1+\varepsilon)B}f(x),K\}
.
\]
It remains to show \(K<\infty\).
To that end take a sequence of John domains \(Q_1,Q_2,\ldots\in\Q\) with \(f_{Q_n}\rightarrow K\) as $n\to\infty$.
That means for every \(n\) we have
\[
\rad{Q_n}\geq\diam{\Lambda\tc{\bll{Q_n}}}/(2\Lambda)\geq\varepsilon r(B)/(4\Lambda)
.
\]
Since there is an \(x_n\in Q_n\cap(1+\varepsilon/4)B\), by definition of a John domain we can conclude that there is a ball \(B_n\subset Q_n\cap(1+\varepsilon/2)B\) with \(\rad{B_n}\geq\varepsilon r(B)/(8\Lambda)\).
By a compactness argument there is a subsequence \(n_1<n_2<\ldots\) such that the balls \((B_{n_k})_k\) converge in \(L^1((1+\varepsilon/2)B)\) to a ball \(B_0\subset(1+\varepsilon/2)B\) with \(r(B_0)\geq\varepsilon r(B)/(8\Lambda)\).
That means for \(C=2^{-1}B_0\) there is a \(k_0\) such that for all \(k\geq k_0\) we have \(C\subset B_{n_k}\subset Q_{n_k}\).
Let \(Q\in\{C,Q_{n_1},Q_{n_2},\ldots\}\).
Since \(f\in L^1_\loc(\Omega)\) with \(\var_\Omega f<\infty\) we know from \cref{rem_medianpoincare} that \(f\in L^1(C)\) and \(f\in L^1(Q_{n_k})\) and thus by Hölder's inequality and \cref{theo_poincare} we obtain for \(k\geq k_0\) that
\begin{align*}
|f_C-f_{Q_{n_k}}|
&\leq
\lm C^{-1}\lVert f-f_{Q_{n_k}}\rVert_{L^1(C)}
\\
&\leq
\lm C^{-(d-1)/d}\lVert f-f_{Q_{n_k}}\rVert_{L^{d/(d-1)}(C)}
\\
&\leq
\lm C^{-(d-1)/d}\lVert f-f_{Q_{n_k}}\rVert_{L^{d/(d-1)}(Q_{n_k})}
\\
&\lesssim_{d,\Lambda}
\lm C^{-(d-1)/d}\var_{Q_{n_k}}f
\\
&\leq
\lm C^{-(d-1)/d}\var_\Omega f
.
\end{align*}
Thus we can conclude
\[
K\lesssim_{d,\Lambda}f_C+r(C)^{-(d-1)}\var_\Omega f<\infty
.
\]
This finishes the proof.
\end{proof}

\begin{cor}
\label{lem_mffinite}
Let \(\Lambda\geq1\), let \(\Omega\subset\mathbb{R}^d\) be an open set and let \(f\in L^1_\loc(\Omega)\) with \(\var_\Omega f<\infty\).
Then \(\Mj_{\Lambda,\Omega} f\in L^{d/(d-1)}_\loc(\Omega)\).
\end{cor}

\begin{proof}
Let \(B\) be a ball with \(\tc B\subset\Omega\).
Then there exists an \(\varepsilon>0\) with \((1+\varepsilon)\tc B\subset\Omega\).
Therefore by \cref{lem_locallylikelocal} we have
\[
\int_B(\Mj_{\Lambda,\Omega} f)^{\f d{d-1}}
\leq
\int_B(\Mj_{\Lambda,\Omega}|f|)^{\f d{d-1}}
\\
\lesssim_{d,\Lambda,\varepsilon}
\lm BK^{\f d{d-1}}
+
\int_B(\Mu_{(1+\varepsilon)B}|f|)^{\f d{d-1}}
.
\]
By \cref{eq_hlmft} we have
\[
\int_B(\Mu_{(1+\varepsilon)B}|f|)^{\f d{d-1}}
\lesssim_d
\int_{(1+\varepsilon)B}|f|^{\f d{d-1}}
\lesssim_d
\lm{(1+\varepsilon)B}|f_{(1+\varepsilon)B}|^{\f d{d-1}}
+
\int_{(1+\varepsilon)B}|f-f_{(1+\varepsilon)B}|^{\f d{d-1}}
.
\]
We conclude the proof observing that \(K<\infty\), \(|f_{(1+\varepsilon)B}|<\infty\) and that by \cref{theo_poincare} we have
\[
\int_{(1+\varepsilon)B}|f-f_{(1+\varepsilon)B}|^{\f d{d-1}}
\lesssim_d
(\var_{(1+\varepsilon)B}f)^{\f d{d-1}}
<
\infty
.
\qedhere
\]
\end{proof}


\begin{lem}[{\cite[Theorem~5.2]{MR3409135}}]
\label{lem_l1approx}
Let \(\Omega\subset\mathbb{R}^d\) be an open set and assume that \(f_1,f_2,\ldots\in L^1_\loc(\Omega)\) are functions with \(\var_\Omega f_n<\infty\) which converge in \(L^1_\loc(\Omega)\) to a function \(f\) as $n\to\infty$.
Then
\[
\var_\Omega f
\leq
\liminf_{n\rightarrow\infty}
\var_\Omega f_n
.
\]
\end{lem}

In \cite[Theorem~5.2]{MR3409135} they assume \(f_n\in\BV(\Omega)\), which is not necessary.

\begin{cor}
\label{cor_lindelofopen}
Every set \(\Q\) of open sets has a countable subset \(\P\subset\Q\) with
\(
\bigcup\P
=
\bigcup\Q
.
\)
\end{cor}

\begin{proof}
Any open set \(U\) is the union of all balls \(B\subset U\).
That means for
\[
\B=\{\tx{ball }B:\exists U\in\Q\ B\subset U\}
\]
we have \(\bigcup\Q=\bigcup\B\) and by \cref{lem_lindelof} there is countable subset \(\C\subset\B\) with \(\bigcup\C=\bigcup\B\).
For each \(B\in\C\) choose one open \(Q\in\Q\) with \(B\subset Q\) and let \(\P\) be the resulting set of open sets \(Q\).
Then \(\P\) is countable and
\[
\bigcup\Q=\bigcup\B=\bigcup\C\subset\bigcup\P\subset\bigcup\Q
.
\]
\end{proof}

\begin{proof}[Proof of \cref{theo_maininfinite}]
By \cref{lem_mffinite} we have \(\M_\Q f\in L^{d/(d-1)}_\loc(\Omega)\) and so it remains to prove
\[
\var_\Omega
\M_\Q f
\lesssim_{d,L,\Lambda}
\var_\Omega f
.
\]
We first consider the case that all stars in \(\Q\) are open.
For every \(q\in\mathbb{Q}\), let
\[
\Q^q=\{Q\in\Q:f_Q>q\}
.
\]
By \cref{cor_lindelofopen} it has a countable subset \(\R_q\) with \(\bigcup\R_q=\bigcup\Q^q\).
Recall that any \(Q\in\Q\) has an \(L\)-decomposition \(\P(Q)\).
For \(Q\in\Q\) inductively define
\begin{align*}
\Q_{Q,0}
&=\{Q\}
,&
\Q_{Q,n+1}
&=
\bigcup\{\P(Q)\cap\Q:Q\in\Q_{Q,n}\}
,&
\Q_Q
&=
\Q_{Q,0}\cup\Q_{Q,1}\cup\ldots
.
\end{align*}
With this definition let
\[
\S
=
\bigcup_{q\in\mathbb{Q}}\bigcup_{Q\in\R_q}
\Q_Q
.
\]
Then \(\S\) is countable and \(L\)-complete because \(\Q\) is \(L\)-complete.
Let \(x\in\Omega\).
Then for every \(\lambda\in\mathbb{R}\) with \(\M_\Q f(x)>\lambda\)
there exists a \(q\in\mathbb{Q}\) with \(\M_\Q f(x)>q\geq\lambda\)
and we have \(f(x)>q\geq\lambda\)
or
\[
x\in\bigcup\Q^q=\bigcup\R_q\subset\bigcup\S^q
.
\]
We can conclude that
\[
\M_\Q f(x)
=
\max
\biggl\{
f(x)
,
\sup_{Q\in\S,\ x\in Q}
\f1{\lm Q}\int_Q f(y)\intd y
\biggr\}
.
\]
Take a sequence \(\S_1\subset\S_2\subset\ldots\) of finite subsets of \(\S\) which are \(L\)-complete
and define
\[
\M_n f(x)
=
\max
\biggl\{
f(x)
,
\max_{Q\in\S_n,\ x\in Q}
\f1{\lm Q}\int_Q f(y)\intd y
\biggr\}
.
\]
Then for every \(x\in\Omega\) we have that
\[
f(x)\leq\M_n f(x)\leq\M_\Q f(x)
\]
and \(\M_n f(x)\) monotonously tends to \(\M_\Q f(x)\) from below.
Let \(B\) be a ball with \(\tc B\subset\Omega\).
Since \(f\in L^1_\loc(\Omega)\) we have \(\int_B|f|<\infty\)
and by \cref{lem_mffinite} we have \(\int_B|\M_\Q f|<\infty\).
So we can conclude by monotone convergence that
\[
\int_B|\M_n f(x)-\M_\Q f(x)|\intd x
\rightarrow0
\]
for \(n\rightarrow\infty\).
It follows from \cref{lem_l1approx} that
\begin{equation}
\label{eq_approx}
\var_\Omega\M_\Q f
\leq
\liminf_{n\rightarrow\infty}
\var_\Omega\M_n f
,
\end{equation}
and it suffices to bound \(\var_\Omega\M_n f\) uniformly.
Since \(\S_n\) is finite we have
\[
\{\M_n f\geq \lambda\}
=
\{f\geq\lambda\}
\cup
\bigcup\{Q\in\S_n:f_Q\geq\lambda\}
\]
and thus by \cref{lem_boundaryofunion} we have
\[
\sm{
\mb{
\{\M_n f\geq \lambda\}
}
\cap\Omega
}
\leq
\smb{
\mb{
\bigcup\{Q\in\S_n:f_Q\geq\lambda\}
}
\setminus
\mc{
\{f\geq\lambda\}
}
}
+
\sm{
\mb{
\{f\geq\lambda\}
\cap\Omega
}
}
.
\]
Using \cref{lem_coareabv} we can conclude from \cref{theo_main} that
\begin{align*}
\var_\Omega\M_n f
&\leq
\int_{-\infty}^\infty
\biggl[
\smb{
\mb{
\bigcup\{Q\in\S_n:f_Q\geq\lambda\}
}
\setminus
\mc{
\{f\geq\lambda\}
}
}
+
\sm{
\mb{
\{f\geq\lambda\}
}
\cap\Omega
}
\biggr]
\intd \lambda
\\
&\lesssim_{d,L,\Lambda}
\int_{-\infty}^\infty
\sm{
\mb{
\{f\geq\lambda\}
}
\cap\Omega
}
\intd \lambda
\\
&=
\var_\Omega f
.
\end{align*}
By \cref{eq_approx} this finishes the proof in the case that all stars in \(\Q\) are open.

Now let \(\Q\) be an \(L\)-complete set \(\Q\) of general \(\Lambda\)-stars.
Then by \cref{lem_smallclosure}\cref{it_opencompletestars} the set \(\ti\Q\) is an \(L\)-complete set of open \(\Lambda\)-stars and thus
\[
\var_\Omega\M_{\ti\Q}\lesssim_{d,L,\Lambda}\var_\Omega f
\]
holds by what we have just shown.
Moreover, by \cref{rem_tiortc} the maximal functions \(\M_\Q f\) and \(\M_{\ti\Q}f\) only differ on a set of Lebesgue measure zero.
That means \(\M_\Q f\) and \(\M_{\ti\Q}f\) have the same variation and so we can conclude the result also in the general case.
\end{proof}

\bibliographystyle{alpha}
\bibliography{bib}

\end{document}